\documentclass[10pt]{amsart}
\usepackage[utf8]{inputenc}
\usepackage[T1]{fontenc}
\usepackage{amsmath,amssymb,amsthm,mathtools,mathrsfs}
\usepackage{enumitem}
\usepackage{xcolor}

\usepackage[hypertexnames=false,hidelinks]{hyperref}
\usepackage[a4paper,hmargin=1.5in,vmargin=1in]{geometry}

\newcommand{\RR}{\mathbb R}
\newcommand{\SSS}{\mathbb S}
\newcommand{\HHH}{\mathbb H}
\newcommand{\QQ}{\mathbb Q}
\newcommand{\ip}[2]{\langle #1,#2\rangle}
\newcommand{\tr}{\operatorname{tr}}
\newcommand{\diver}{\operatorname{div}}
\newcommand{\Id}{\operatorname{Id}}
\newtheorem{theorem}{Theorem}[section]
\newtheorem{maintheorem}{Theorem}

\newtheorem{lemma}[theorem]{Lemma}
\newtheorem{proposition}[theorem]{Proposition}
\newtheorem{corollary}[theorem]{Corollary}
\theoremstyle{remark}
\newtheorem{remark}[theorem]{Remark}
\title[An Abresch--Rosenberg type tensor]{An Abresch--Rosenberg type tensor and a gap theorem for CMC hypersurfaces in $\SSS^n\times\RR$ and $\HHH^n\times\RR$}
\author{Hil{\'a}rio Alencar}
\address{Instituto de Matem{\'a}tica, Universidade Federal de Alagoas, 57072-900, Macei{\'o}, AL, Brazil}
\email{hilario@mat.ufal.br}
\author{M\'arcio Batista}
\address{CPMAT -- Instituto de Matem{\'a}tica, Universidade Federal de Alagoas, 57072-900, Macei{\'o}, AL, Brazil}
\email{mhbs@mat.ufal.br}
\author{Harold Rosenberg}
\address{13 rue du Faubourg Montmartre, 75009 Paris, France}
\email{rosenbergharold@gmail.com}
\date{\today}
\begin{document}
\begin{abstract}
We construct a symmetric trace-free tensor for constant-mean-curvature hypersurfaces in $\SSS^n\times\RR$ and $\HHH^n\times\RR$ that extends the Abresch--Rosenberg tensor. Within its natural class, it is characterized by the linearity of its Codazzi defect, and its complete connected zero set consists exactly of the rotational CMC spheres. An exact integrated Simons identity yields explicit gap theorems, valid for every $h>1$ in $\HHH^n\times\RR$ and every $h>(n\int_0^{\pi/2}\sin^{n-1}r\,dr)^{-1}$ in $\SSS^n\times\RR$. We compare the pinching thresholds with the pointwise algebraic level that implies convexity. In $\HHH^n\times\RR$ the explicit pinching conditions imply strict convexity, so there the method gives an independent proof with explicit constants. In $\SSS^n\times\RR$ the pinching region exceeds the convexity level for $n\in\{3,4\}$ and every admissible $h$, and for $n\geq5$ whenever $h^2\geq n/(6(n+2))$; there the gap theorem does not require convexity to be known in advance. We also determine second-order large-$h$ asymptotics of the thresholds.
\end{abstract}
\maketitle

\section{Introduction}
Hopf's holomorphic quadratic differential provides a fundamental rigidity mechanism for constant-mean-curvature (CMC) two-spheres. In the products $\SSS^2\times\RR$ and $\HHH^2\times\RR$, Abresch and Rosenberg replaced it by a holomorphic differential adapted to the ambient geometry \cite{AR}. The associated symmetric trace-free tensor was studied through Simons-type equations in \cite{Batista}. For a CMC surface with normalized mean curvature $h$, shape operator $A$, and vertical tangent field $T$, it takes the form
\begin{equation}\label{old}
S_2=2h(A-h\Id)-c\left(T\otimes T-\frac{|T|^2}{2}\Id\right),\qquad c\in\{-1,1\}.
\end{equation}
Its Codazzi property is special to dimension two. The holomorphic-differential approach was also extended to surfaces with parallel mean-curvature vector in higher-dimensional ambient products \cite{ACT,FR}; in those works the surface itself still has dimension two. Espinar and Trejos \cite{ET} showed that, on $H$-surfaces in the homogeneous spaces $\mathbb E(\kappa,\tau)$ with $\tau\neq0$, the Abresch--Rosenberg shape operator forms a Codazzi pair; they derived a Simons-type equation and a gap theorem. This is again a two-dimensional setting, in which the Codazzi property is available. For submanifolds of arbitrary dimension $m\geq3$ with parallel mean curvature vector in $M^n(c)\times\RR$, Fetcu and Rosenberg \cite{FR13} proved a Simons-type formula for the second fundamental form and its traceless part. That tensor carries no correction by the vertical field, and their gap theorems conclude that the submanifold is totally umbilical in a slice $M^{m+1}(c)$. It does not vanish on the rotational CMC spheres, which do not lie in slices, and therefore cannot characterize them. More recently, Fetcu and Nardulli \cite{FN} computed Simons-type equations for parallel-mean-curvature submanifolds of complex space forms and of their products with $\RR$, using the complex, respectively cosymplectic, structure; their tensors are built from the shape operator of a parallel normal field, without an Abresch--Rosenberg correction. We are not aware of a tensor adapted to the rotational CMC spheres of $\QQ^n_c\times\RR$ for $n\geq3$.

Rotational CMC spheres, on the other hand, exist in $\QQ_c^n\times\RR$ in every dimension $n\geq2$, where $\QQ_1^n=\SSS^n$ and $\QQ_{-1}^n=\HHH^n$. Their geometry and isoperimetry are discussed in \cite{HH,Pedrosa,LEN}. We use the explicit profile from \cite{LEN} to construct a dimension-dependent replacement of \eqref{old}. This choice matters: the literal trace-free extension of \eqref{old}, with $|T|^2/2$ replaced by $|T|^2/n$, does not vanish on the rotational spheres when $n>2$.

The central question is whether the two-dimensional tensor can be replaced by one adapted to the rotational model while retaining a differential structure useful for rigidity. We answer this within the class of tensors in \eqref{Bdef}: the model coefficient is also characterized by cancellation of the tensor-independent part of the Codazzi defect. The resulting defect is linear in the tensor and can therefore be retained, and estimated, in an integrated Simons identity. This construction and characterization, together with the classification of the complete zero set, provide the common mechanism for the gap theorems below.

\subsection*{The adapted tensor and its differential identity}
Throughout, $h=\tr A/n>0$. We assume the admissible range
\begin{equation}\label{range}
h>0\quad(c=1),\qquad nh>n-1\quad(c=-1).
\end{equation}
The second inequality is the existence range of the compact hyperbolic-product reference sphere. Section \ref{profile} defines a smooth scalar function $k_{n,c,h}:[0,1]\to\RR$ solely from its profile. For an arbitrary CMC immersion of the same mean curvature, put
\begin{equation}\label{Bdef}
B=\mathscr S_{n,c,h}:=2h(A-h\Id)-k_{n,c,h}(|T|^2)
\left(T\otimes T-\frac{|T|^2}{n}\Id\right).
\end{equation}

\begin{maintheorem}\label{main-equality}
If $n=2$, then $k_{2,c,h}\equiv c$ and $B=S_2$ on every CMC surface. For every $n\geq2$, the rotational sphere in the range \eqref{range} satisfies $B=0$. Conversely, any connected, complete, two-sided CMC immersion $x:\Sigma^n\to\QQ_c^n\times\RR$ with $n\geq2$, mean curvature $h$ in \eqref{range}, and $B\equiv0$ is congruent to that rotational sphere.
\end{maintheorem}
The converse permits immersed hypersurfaces without an embeddedness assumption. Its local part identifies the horizontal levels as concentric geodesic spheres; completeness enters only after analytic continuation has placed the entire immersion in one reference sphere. For $n=2$ it gives, in the range \eqref{range}, a direct proof of the part of the classification in \cite{AR} concerning complete surfaces with vanishing Abresch--Rosenberg differential.

For $n>2$ the adapted tensor generally is not Codazzi. Its defect, however, is linear in $B$, with no term independent of $B$ (Proposition \ref{defect}). Proposition \ref{canonical} characterizes $k_{n,c,h}$ by this cancellation among coefficients continuous on $[0,1]$ and of class $C^1$ on $(0,1]$. The cancellation is required as a formal identity for every $s\in(0,1]$, rather than only at values realized by one immersion. Thus the characterization concerns the prescribed class \eqref{Bdef}. The linear defect leads to the following exact integrated identity.

\begin{maintheorem}\label{main-identity}
Let $\Sigma^n$ be a closed, two-sided CMC hypersurface in the range \eqref{range}. Write $D(X,Y)=(\nabla_XB)Y-(\nabla_YB)X$. In a local orthonormal frame diagonalizing $B$, let $\beta_i$ be its eigenvalues and $K_{ij}$ the sectional curvature of the $ij$-plane. Then
\begin{equation}\label{main-int}
0=\int_\Sigma\left\{|\nabla B|^2+
\sum_{i<j}K_{ij}(\beta_i-\beta_j)^2-
\frac12|D|^2-|\diver B|^2\right\}.
\end{equation}
Here $|D|^2=\sum_{i,j,k}D_{ijk}^2$, with $D_{ijk}=\ip{D(e_i,e_j)}{e_k}$, is summed over all ordered index triples. The explicit formula for $D$ is \eqref{linearD}. Both $D$ and $\diver B$ are of order zero in $B$; see \eqref{divB}.
\end{maintheorem}

\subsection*{Quantitative rigidity}
The next result makes all constants explicit and holds for every $h$ in \eqref{range} for which the constant $a_0$ below is positive. The numbers $a_0=a_0(n,c,h)$ and $a_1=a_1(n,c,h)$ are defined in \eqref{constants}; they are elementary expressions in $n$ and $h$.

\begin{maintheorem}\label{main-gap}
Let $n\geq3$, let $h$ be in \eqref{range} with $a_0(n,c,h)>0$, and put
\begin{equation}\label{xplus}
x_+(n,c,h)=2h^2\left(\sqrt{a_1^2+a_0/h^2}-a_1\right).
\end{equation}
If a connected, closed, two-sided CMC immersion $\Sigma^n\to\QQ_c^n\times\RR$ has normalized mean curvature $h$ and
\begin{equation}\label{pinch}
\sup_\Sigma|\mathscr S_{n,c,h}|\leq x_+(n,c,h),
\end{equation}
then it is congruent to the rotational CMC sphere of mean curvature $h$. Moreover $a_0>0$ whenever $h^2\geq8$, and
\[
\lim_{h\to\infty}\frac{x_+(n,c,h)}{h^2}=2\sqrt{\frac{n}{n-1}}\qquad(c=\pm1).
\]
\end{maintheorem}
Equivalently, $x_+=2a_0\big/\big(\sqrt{a_1^2+a_0/h^2}+a_1\big)$; this form avoids the subtraction of nearly equal numbers. The positivity of $a_0$ for $h^2\geq8$ is Lemma \ref{a0-positive}. The threshold $x_+$ estimates the curvature terms separately. Minimizing them jointly, using Okumura's inequality with one eigenvalue prescribed and its equality case \cite[Lemma (2.6)]{AdC}, with $k(s)$ relaxed to an interval independent of $s$, gives a threshold $x_{\rm opt}\geq x_+$ for which the same conclusion holds under the strict pinching $\sup_\Sigma|\mathscr S_{n,c,h}|<x_{\rm opt}$ (Theorem \ref{opt-gap}). For $n=2$ the corresponding gap theorem for the Abresch--Rosenberg tensor is due to Batista \cite{Batista}; for this reason Theorems \ref{main-gap}--\ref{main-asymp} are stated for $n\geq3$, while Theorems \ref{main-equality} and \ref{main-identity} hold for $n\geq2$. The condition $a_0>0$ is an explicit inequality in $n$ and $h$, so it can be checked directly for given values with $h^2<8$.

The pinching \eqref{pinch} forces strict convexity (Proposition \ref{convex-forced}). Consequently, the classification of de Lima, Manfio and dos Santos \cite[Theorem 17]{LMS} also gives the conclusion of Theorem \ref{main-gap}. Here the adapted tensor provides an independent integral proof and an explicit sufficient pinching condition. In $\SSS^n\times\RR$ the uniform convexity estimate also implies compactness under completeness.

\begin{corollary}\label{complete-sphere}
Let $n\geq3$ and let $h>0$ satisfy $a_0(n,1,h)>0$. A connected, complete, two-sided CMC immersion $\Sigma^n\to\SSS^n\times\RR$ with normalized mean curvature $h$ and $\sup_\Sigma|\mathscr S_{n,1,h}|\leq x_+(n,1,h)$ is congruent to the rotational CMC sphere of mean curvature $h$.
\end{corollary}

In $\HHH^n\times\RR$ the same argument only gives $K_\Sigma\geq\delta^2-1$ for a positive lower bound $\delta$ of the principal curvatures. Under a stronger pinching we obtain a uniform statement in both products.

\begin{corollary}\label{complete}
Let $n\geq3$ and $h\geq2\sqrt2$. A connected, complete, two-sided CMC immersion $\Sigma^n\to\QQ_c^n\times\RR$ with normalized mean curvature $h$ and $\sup_\Sigma|\mathscr S_{n,c,h}|\leq h^2/4$ is congruent to the rotational CMC sphere of mean curvature $h$.
\end{corollary}
For $c=1$, Corollary \ref{complete} is contained in Corollary \ref{complete-sphere}. Both corollaries follow by first obtaining compactness from a uniform positive lower bound for sectional curvature and then applying Theorem \ref{main-gap}.

The constants $a_0$ and $a_1$ come from uniform bounds on $k$ and $sk'$ (Lemma \ref{bounds}), and they require $h$ to be moderately large. Using instead the sign structure of the model coefficient (Lemma \ref{model-signs}), an exact block computation of the Codazzi defect (Lemma \ref{sharp-defect}) and the sectional curvatures of the reference sphere, we obtain a second gap theorem on a much larger range of $h$. Put
\[
h_{\SSS}(n)=\Big(n\int_0^{\pi/2}\sin^{n-1}r\,dr\Big)^{-1}.
\]
The functions $q_{n,c,h}(s)$ and $\mathcal P_{n,c,h}(s)$ in \eqref{qPdef} below depend only on the rotational profile. In particular, the following threshold is an effective one-variable quantity, without a large-mean-curvature assumption.

\begin{maintheorem}\label{main-sharp}
Let $n\geq3$ and assume
\begin{equation}\label{gap-range}
h>h_{\SSS}(n)\quad\text{for } c=1,\qquad h>1\quad\text{for } c=-1.
\end{equation}
Then $q_{n,c,h}(s)>0$ for every $s\in[0,1]$ and
\begin{equation}\label{radius}
R_*(n,c,h)=\min_{0\leq s\leq1}2h^2\left(\sqrt{\mathcal P_{n,c,h}(s)^2+q_{n,c,h}(s)/h^2}-\mathcal P_{n,c,h}(s)\right)>0.
\end{equation}
Every connected, closed, two-sided CMC immersion $\Sigma^n\to\QQ_c^n\times\RR$ with normalized mean curvature $h$ satisfying
\begin{equation}\label{pinch-sharp}
\sup_\Sigma|\mathscr S_{n,c,h}|<R_*(n,c,h)
\end{equation}
is congruent to the rotational CMC sphere of mean curvature $h$.
\end{maintheorem}

The ranges \eqref{gap-range} are exactly the ranges on which $q_{n,c,h}$ is positive everywhere: at the boundary it vanishes at the poles in the hyperbolic product and at the equator in the spherical product (Lemma \ref{positive-q}). The spherical range is also exactly the strict-convexity range of the reference sphere (Remark \ref{sharp-convexity}). For numerical evaluation, the equivalent expression
\[
R_*=\min_{s\in[0,1]}\frac{2q_{n,c,h}(s)}{\sqrt{\mathcal P_{n,c,h}(s)^2+q_{n,c,h}(s)/h^2}+\mathcal P_{n,c,h}(s)}
\]
avoids subtraction of nearly equal numbers. Theorems \ref{main-gap} and \ref{main-sharp} use different lower bounds in \eqref{main-int}; no general ordering of $x_+$ and $R_*$ is asserted. Theorem \ref{main-gap} has elementary closed-form constants and admits the endpoint $\sup_\Sigma|B|=x_+$, while Theorem \ref{main-sharp} uses the geometry of the reference sphere and covers the whole range \eqref{gap-range}.

\subsection*{Comparison with convexity}
Let $X_{\rm conv}(n,c,h)$ denote the pointwise algebraic level below which the pinching forces all principal curvatures to be positive (Proposition \ref{convex-criterion}). Comparing $R_*$ with $X_{\rm conv}$ decides whether Theorem \ref{main-sharp} goes beyond the classification of convex hypersurfaces in \cite[Theorem 17]{LMS}.

\begin{maintheorem}\label{main-compare}
Let $n\geq3$ and let $h$ satisfy \eqref{gap-range}.
\begin{enumerate}
\item[$($i$)$ ] If $c=-1$, then $R_*(n,-1,h)<X_{\rm conv}(n,-1,h)$ for every $h>1$. Therefore every hypersurface satisfying the hypotheses of Theorem \ref{main-sharp} in $\HHH^n\times\RR$ is strictly convex.
\item[$($ii$)$] If $c=1$, then $R_*(n,1,h)>X_{\rm conv}(n,1,h)$ provided either $n\in\{3,4\}$, or $n\geq5$ and $h^2\geq n/(6(n+2))$. In these cases the pinching region of Theorem \ref{main-sharp} contains algebraic data $(B,T)$ for which the operator determined by \eqref{Bdef} has a negative eigenvalue.
\end{enumerate}
\end{maintheorem}

Part ($i$) shows that in $\HHH^n\times\RR$ neither Theorem \ref{main-gap} nor Theorem \ref{main-sharp} gives rigidity beyond \cite[Theorem 17]{LMS}, and the Alexander--Hsiang--Hsiang route of Remark \ref{convexity} is also available there; in the hyperbolic product the contribution of the method is an independent proof through the adapted tensor, with explicit constants. Part ($ii$) is what distinguishes the method from convexity in $\SSS^n\times\RR$. Its proof rests on an exact identity at the equator of the reference sphere (Lemma \ref{equator}). For $n\geq5$ and $h_{\SSS}(n)<h<\sqrt{n/(6(n+2))}$ the inequality $R_*>X_{\rm conv}$ is supported by numerical integration of the profile (Remark \ref{window}) but is not proved here. By Theorem \ref{main-asymp}, the gain $R_*-X_{\rm conv}$ is a window of width $O(1)$ at the scale $2\sqrt{n/(n-1)}\,h^2$.

\subsection*{Asymptotics}
For fixed $n$ and $c$, all thresholds have the same leading scale $\lambda h^2$, $\lambda=2\sqrt{n/(n-1)}$; the next-order terms are as follows.

\begin{maintheorem}\label{main-asymp}
Let $n\geq3$, $c=\pm1$ and $\lambda=2\sqrt{n/(n-1)}$. As $h\to\infty$:
\begin{enumerate}
\item[$($a$)$] $x_{\rm opt}(n,c,h)=\lambda h^2+\mu_c+o(1)$, where
\[
\mu_1=\frac{2(n^2-n-4)}{(n+2)\sqrt{n(n-1)}}>0,\qquad \mu_{-1}=-2\sqrt{\tfrac{n-1}{n}};
\]
\item[$($b$)$] with $\rho_1=\dfrac{2(n-4)(n-1)}{(n+2)\sqrt{n(n-1)}}$,
\begin{align*}
R_*(n,1,h)&=\lambda h^2+\rho_1+O(h^{-2}),\\
X_{\rm conv}(n,1,h)&=\lambda h^2-\frac{2\sqrt n}{(n+2)\sqrt{n-1}}+O(h^{-2}),\\
R_*(n,-1,h)&=\lambda h^2+\mu_{-1}+O(h^{-2}),\\
X_{\rm conv}(n,-1,h)&=\lambda h^2-\frac{2\sqrt{n(n-1)}}{n+2}+O(h^{-2});
\end{align*}
\item[$($c$)$] consequently, for $c=1$,
\[
R_*-X_{\rm conv}\to\frac{2(n-2)^2}{(n+2)\sqrt{n(n-1)}},\qquad x_{\rm opt}-R_*\to\frac{8(n-2)}{(n+2)\sqrt{n(n-1)}},
\]
while for $c=-1$,
\[
X_{\rm conv}-R_*\to\frac{4\sqrt{n-1}}{(n+2)\sqrt n},\qquad x_{\rm opt}-R_*\to0 .
\]
\end{enumerate}
In particular, for $c=1$ and all sufficiently large $h$, the pinching region of Theorem \ref{opt-gap} also contains algebraic data $(B,T)$ for which the operator determined by \eqref{Bdef} has a negative eigenvalue (Proposition \ref{nonconvex}).
\end{maintheorem}

\begin{remark}\label{algebraic-remark}
Theorem \ref{main-compare}($ii$), Theorem \ref{main-asymp} and Propositions \ref{convex-criterion}($ii$) and \ref{nonconvex} concern pointwise algebraic data $(B,T)$. They show that the pinching conditions are not contained in the region where a pointwise estimate based only on $B$ and $T$ forces convexity. These data are not asserted to satisfy the compatibility equations of an immersion, and no nonconvex CMC hypersurface satisfying the pinching is claimed to exist: every closed immersion satisfying the hypotheses of Theorems \ref{main-gap} or \ref{main-sharp} is identified with the strictly convex reference sphere. Likewise, the positivity range of $q_{n,c,h}$ is exact only for the estimate of Section \ref{sharp-section}, and $x_{\rm opt}$ optimizes a relaxed algebraic estimate; neither is claimed to be an optimal geometric rigidity threshold.
\end{remark}

\subsection*{Scope and comparison of the results}
Table \ref{threshold-summary} records the hypotheses and endpoint conventions of the three gap criteria. In every row, $n\geq3$, $h$ lies in \eqref{range}, and the immersion is connected, closed and two-sided. The complete versions established here are Corollaries \ref{complete-sphere} and \ref{complete}.

\begin{table}[htbp]
\caption{Gap criteria for the adapted tensor.}\label{threshold-summary}
\centering
\small
\renewcommand{\arraystretch}{1.25}
\begin{tabular}{@{}p{0.16\textwidth}p{0.26\textwidth}p{0.21\textwidth}p{0.27\textwidth}@{}}
\hline
Result & Additional condition & Pinching & Comparison established \\
\hline
Theorem \ref{main-gap} & $a_0>0$ & $\sup_\Sigma|B|\leq x_+$ & Forces strict convexity \\
Theorem \ref{main-sharp} & $h>h_{\SSS}(n)$ if $c=1$; $h>1$ if $c=-1$ & $\sup_\Sigma|B|<R_*$ & $R_*<X_{\rm conv}$ if $c=-1$; $R_*>X_{\rm conv}$ if $c=1$ (Theorem \ref{main-compare}) \\
Theorem \ref{opt-gap} & $x_{\rm opt}>0$ & $\sup_\Sigma|B|<x_{\rm opt}$ & $x_{\rm opt}\geq x_+$ if $a_0>0$ \\
\hline
\end{tabular}
\end{table}

Theorem \ref{main-compare} compares $R_*$ with $X_{\rm conv}$ for all $h>1$ in the hyperbolic product and, in the spherical product, on the range stated there. Theorem \ref{main-asymp} gives, in addition, $x_{\rm opt}>R_*>X_{\rm conv}$ in the spherical product for all sufficiently large $h$; in the hyperbolic product it gives $x_{\rm opt}-R_*\to0$, without asserting an eventual ordering of $x_{\rm opt}$ and $R_*$. The non-strict endpoint in Theorem \ref{main-gap} follows from a positive remainder in its estimate; no such endpoint conclusion is claimed for the other two criteria.

These comparisons separate the contribution of the tensor method from the convexity classification in \cite{LMS}. Once strict convexity is known, that classification supplies rotational rigidity. The present construction instead characterizes a tensor, classifies its complete zero set, and obtains sufficient pinching conditions from its differential identity. In the spherical product, the improved conditions are not contained in the pointwise algebraic region forcing convexity (Remark \ref{algebraic-remark}). The practical gain is that Theorem \ref{main-sharp} applies without knowing in advance that $\Sigma$ is convex. We do not claim a closed CMC hypersurface covered by Theorem \ref{main-sharp} and not by \cite{LMS}: every hypersurface satisfying its hypotheses is the strictly convex reference sphere.

The curvature-commutation argument in \eqref{main-int} is related to Cheng and Yau's formula \cite{CY}, with all terms arising from the Codazzi defect retained. The algebraic estimate for Theorem \ref{main-gap} follows the method of \cite{AdC}, including Okumura's inequality \cite{Okumura}. Sections \ref{sharp-section} and \ref{opt-section} refine this estimate by using the model geometry and joint minimization, respectively.

\textbf{The paper is organized as follows}. Section \ref{geometry} fixes conventions, and Section \ref{profile} constructs the model coefficient. Section \ref{equal-section} classifies the complete zero set of $B$. Section \ref{simons-section} computes its Codazzi defect and proves the integrated Simons identity. Section \ref{gap-section} establishes the explicit gap theorem and its completeness consequences. Section \ref{sharp-section} proves the gap theorem throughout the positivity range and compares its threshold with convexity, proving Theorem \ref{main-compare}. Section \ref{opt-section} develops the joint algebraic estimate and the asymptotic comparisons.

\section{Geometry and sign conventions}\label{geometry}
Our curvature convention, for orthonormal $X,Y$, is
\begin{gather*}
R(X,Y)Z=\nabla_X\nabla_YZ-\nabla_Y\nabla_XZ-\nabla_{[X,Y]}Z,\\ K(X,Y)=\ip{R(X,Y)Y}{X}.
\end{gather*}
Throughout, \emph{closed} means compact without boundary. Two-sidedness is a genuine assumption for immersions: a closed immersed hypersurface of a simply connected manifold need not be two-sided, as immersions of the projective plane in $\RR^3$ show. Let $\xi=\partial_t$ be the parallel vertical unit field. For a chosen unit normal $N$, set
\[
\xi=T+\nu N,\quad |T|^2+\nu^2=1,\quad A X=-\bar\nabla_XN,\quad h=\frac1n\tr A.
\]
For a CMC immersion in $\QQ_c^n\times\RR$, the compatibility identities are
\begin{align}
\nabla_XT&=\nu AX,& X(\nu)&=-\ip{AX}{T},\label{compat}\\
(\nabla_XA)Y-(\nabla_YA)X
&=-c\nu\big(\ip XT Y-\ip YT X\big).\label{codA}
\end{align}
In particular, with $s=|T|^2$,
\begin{equation}\label{grads}
Xs=2\nu\ip{AX}{T},\qquad \nabla s=2\nu AT.
\end{equation}
The Gauss equation for an orthonormal basis $X,Y$ of a tangent plane is
\begin{equation}\label{gauss}
K_\Sigma(X,Y)=c(1-\ip XT^2-\ip YT^2)
+\ip{AX}{X}\ip{AY}{Y}-\ip{AX}{Y}^2.
\end{equation}
The first term is at least $-1$ in either product.

For completeness, here are the derivations. Write $V^*=V-\ip V\xi\xi$ for the horizontal part. The curvature tensor of $\QQ_c^n\times\RR$ is $\bar R(X,Y)Z=c\big(\ip{Y^*}{Z^*}X^*-\ip{X^*}{Z^*}Y^*\big)$.
\begin{enumerate}[label=(\alph*)]
\item Since $\xi$ is parallel, $0=\bar\nabla_X(T+\nu N)=\nabla_XT+\ip{AX}{T}N+X(\nu)N-\nu AX$. The tangent and normal parts give \eqref{compat}.
\item Differentiating $AY=-\bar\nabla_YN$ gives $(\nabla_XA)Y-(\nabla_YA)X=-(\bar R(X,Y)N)^\top$. Since $\ip{Y^*}{N^*}=-\nu\ip YT$ and $(X^*)^\top=X-\ip XT T$, this gives \eqref{codA}.
\item For orthonormal $X,Y$ we have $|X^*|^2=1-\ip XT^2$ and $\ip{X^*}{Y^*}=-\ip XT\ip YT$. Hence $\ip{\bar R(X,Y)Y}{X}=c(1-\ip XT^2-\ip YT^2)$, and the Gauss equation gives \eqref{gauss}.
\end{enumerate}

\section{The model coefficient}\label{profile}
Set
\[
f_c(r)=\begin{cases}\sin r&c=1,\\ \sinh r&c=-1,\end{cases}\quad
m(r)=\frac{f_c'(r)}{f_c(r)},\quad
w(r)=\frac{\int_0^r f_c(u)^{n-1}\,du}{f_c(r)^{n-1}}.
\]
Here $r\in(0,\pi)$ for $c=1$ and $r\in(0,\infty)$ for $c=-1$. Differentiation gives
\begin{equation}\label{odew}
w'=1-(n-1)mw,\qquad m'=-c-m^2.
\end{equation}
The function $w$ increases strictly from $0$; it tends to infinity at $\pi$ for $c=1$ and to $1/(n-1)$ at infinity for $c=-1$ \cite[Lemma 1]{LEN}. For completeness we note the elementary proofs. For $c=1$ the numerator of $w$ tends to $\int_0^\pi\sin^{n-1}>0$ and the denominator to $0$. For $c=-1$, l'H\^opital's rule gives the limit $1/(n-1)$. Moreover $w<1/(n-1)$, because $\frac{d}{dr}\frac{\sinh^{n-1}r}{n-1}=\sinh^{n-2}r\cosh r>\sinh^{n-1}r$ and both sides vanish at $0$. Strict monotonicity, $w'>0$, is proved in Lemma \ref{bounds}(c) for $n\geq3$; for $n=2$ one has $w=\tan(r/2)$ or $w=\tanh(r/2)$. Hence, in the range \eqref{range}, a unique $r_0$ solves $nhw(r_0)=1$. The lower hemisphere of the rotational sphere is a radial graph whose slope variable and principal curvatures are \cite[\S4]{LEN}
\begin{equation}\label{model}
\begin{gathered}
\rho(r)=nhw(r),\qquad z'(r)=\frac{\rho(r)}{\sqrt{1-\rho(r)^2}},\\
a=\rho'=nhw',\qquad b=\rho m=nhwm,
\end{gathered}
\end{equation}
where $a$ has meridian multiplicity one, $b$ orbit multiplicity $n-1$, $nh=a+(n-1)b$, and $|T|=\rho$. Reflection in the upper horizontal level completes the sphere. We use the orientation for which $h>0$.

These formulas follow directly. Write the lower hemisphere as $\{(x_0,z):\ d(x_0,p)=r,\ z=Z(r)\}$ and let $\rho=\sin\vartheta$, where $\vartheta\in[0,\pi/2)$ is the angle between the meridian and the horizontal, so that $Z'=\rho/\sqrt{1-\rho^2}$. Take the upward normal $N=-\sin\vartheta\,\partial_r+\cos\vartheta\,\partial_z$.
\begin{enumerate}[label=(\alph*)]
\item Along the meridian, parametrized by arclength $\varsigma$ in the flat $(r,z)$-plane, $dN/d\varsigma=-(d\vartheta/d\varsigma)e_1$. So $a=d\vartheta/d\varsigma=\cos\vartheta\,d\vartheta/dr=\rho'$.
\item For $X$ tangent to the geodesic sphere, $\bar\nabla_XN=-\sin\vartheta\,\nabla^{\QQ_c}_X\partial_r=-\rho mX$. So $b=\rho m$.
\end{enumerate}
Then $nh=\rho'+(n-1)m\rho=(f_c^{n-1}\rho)'/f_c^{n-1}$ with $\rho(0)=0$, which integrates to $\rho=nhw$. Finally $\nu=\ip N\xi=\cos\vartheta$, so $|T|=\sin\vartheta=\rho$.

For $s\in(0,1]$, let $r(s)\in(0,r_0]$ be uniquely determined by $s=(nhw(r))^2$. Define
\begin{equation}\label{defk}
\begin{aligned}
k_{n,c,h}(s)=\frac{2h(a-b)}s
&=\frac{2}{nw(r(s))^2}-\frac{2m(r(s))}{w(r(s))}\\
&=\frac{2[1-nm(r(s))w(r(s))]}{nw(r(s))^2}.
\end{aligned}
\end{equation}
Thus $k_{n,c,h}=F_c\circ r$, where
\begin{equation}\label{Fdef}
F_c(r)=\frac{2[1-nm(r)w(r)]}{nw(r)^2}
\end{equation}
does not depend on $h$; the mean curvature enters only through $r(s)$.

\begin{lemma}\label{regularity}
The function $k_{n,c,h}$ extends real analytically through $s=0$ and smoothly to $s=1$, and
\begin{equation}\label{kzero}
k_{n,c,h}(0)=F_0:=\frac{2nc}{n+2}.
\end{equation}
For $n=2$ it is identically $c$ on $[0,1]$.
\end{lemma}
\begin{proof}
Near zero,
\begin{align*}
f_c(r)&=r-\frac c6r^3+O(r^5),\\
w(r)&=\frac rn\left(1+\frac{c(n-1)}{3(n+2)}r^2+O(r^4)\right),\\
m(r)&=\frac1r-\frac c3r+O(r^3).
\end{align*}
Consequently $1-nmw=cr^2/(n+2)+O(r^4)$, yielding \eqref{kzero}. In fact $w(r)/r$ and $rm(r)$ are analytic in $r^2$. Cancellation in the final expression of \eqref{defk} makes $k$ analytic in $r^2$, while $s=h^2r^2+O(r^4)$ has a local analytic inverse in $r^2$. At $s=1$, $r_0>0$ and $ds/dr=2(nh)^2ww'>0$.

For $n=2$, put $I(r)=\int_0^r f_c(u)\,du=(1-f_c'(r))/c$. Since $w=I/f_c$ and $(f_c')^2+cf_c^2=1$,
\[
k_{2,c,h}=\frac{f_c^2-2f_c'I}{I^2}
=\frac{(1-f_c')^2/c}{I^2}=c.
\]
\end{proof}

The coefficient satisfies a Riccati equation, which gives explicit bounds. Put
\begin{equation}\label{Kn}
\begin{gathered}
\mathcal K_{n,1}=\frac{2n}{n+2},\qquad \mathcal K_{n,-1}=\frac{2(n-1)}{n},\\
\theta_{n,c}(h)=\begin{cases}1,&c=1,\\[2pt] 1-\left(\dfrac{n-1}{nh}\right)^2,&c=-1,\end{cases}
\end{gathered}
\end{equation}
and
\begin{equation}\label{sigma}
\sigma_{n,c}(h)=\frac{(n-2)\mathcal K_{n,c}^2}{4n\,h^2\,\theta_{n,c}(h)}.
\end{equation}
Note that $\theta_{n,-1}(h)>0$ exactly in the range \eqref{range}, and $\mathcal K_{n,c}<2$.

\begin{lemma}\label{bounds}
Let $n\geq3$. Then:
\begin{enumerate}
\item[(a)] $F_c$ satisfies
\begin{equation}\label{riccati}
F_c'=\frac2w\left[c-\frac{n+2}{2n}F_c-\frac{n-2}{4}w^2F_c^2\right].
\end{equation}
\item[(b)] For $c=1$, $F_1>0$ on $(0,\pi)$. For $c=-1$, $F_{-1}>-2(n-1)/n$ on $(0,\infty)$.
\item[(c)] For $c=1$, $w'>1/n$ on $(0,\pi)$. For $c=-1$, $w'>(1-(n-1)^2w^2)/n>0$ on $(0,\infty)$. Hence $w'\geq\theta_{n,c}(h)/n$ on $(0,r_0]$.
\item[(d)] $F_c<F_0$ for $r>0$. Consequently $|k_{n,c,h}|\leq \mathcal K_{n,c}$ on $[0,1]$; more precisely $0<k<\mathcal K$ for $c=1$ and $s>0$, and $|k|<\mathcal K$ for $c=-1$. Moreover, with $Q=n(n-2)\mathcal K_{n,c}^2/(2(n+2))$, one has $-Qw^2\leq F_c-F_0$, and the bracket in \eqref{riccati} has absolute value at most $(n-2)\mathcal K_{n,c}^2w^2/4$.
\item[(e)] For every $h$ in \eqref{range} and $s\in(0,1]$,
\begin{equation}\label{skprime}
|s\,k_{n,c,h}'(s)|\leq\sigma_{n,c}(h).
\end{equation}
\item[(f)] $F_c'<0$ on its whole domain, $(0,\pi)$ for $c=1$ and $(0,\infty)$ for $c=-1$; this does not involve $h$. Thus, $k_{n,c,h}'<0$ on $(0,1]$ for every admissible $h$.
\end{enumerate}
\end{lemma}
\begin{proof}
The items are proved in the order (a), (b), (c), (d); inside (d) we first prove $F<F_0$, then deduce $|F|\leq \mathcal K$, and only then prove $F-F_0\geq-Qw^2$. No item uses a later one.

(a) Put $u=1-nmw$, so $F=2u/(nw^2)$. By \eqref{odew},
$u'=-n(m'w+mw')=n(cw-mu)$. Hence
\[
F'=\frac{2}{nw^3}\big(w\,u'-2uw'\big)
=\frac{2}{nw^3}\big[ncw^2-2u+(n-2)muw\big].
\]
Now substitute $mw=(1-u)/n$ and $u=nw^2F/2$; this gives \eqref{riccati}. (For $n=2$ it reduces to $F'=2(c-F)/w$, which is consistent with $F\equiv c$.)

(b), $c=1$. $F_1(0)=F_0>0$. If $r_1$ were the first zero of $F_1$, then $F_1'(r_1)\leq0$, whereas \eqref{riccati} gives $F_1'(r_1)=2/w>0$. So $F_1>0$.

(b), $c=-1$. Recall that $\mathcal K_{n,-1}=2(n-1)/n$. Since $n\geq3$, $F_0=-2n/(n+2)>-\mathcal K_{n,-1}$. At a point where $F_{-1}=-\mathcal K_{n,-1}$, the bracket in \eqref{riccati} equals
\begin{align*}
&-1+\frac{(n+2)(n-1)}{n^2}-\frac{(n-2)(n-1)^2}{n^2}w^2=\frac{(n-2)\big(1-(n-1)^2w^2\big)}{n^2}>0,
\end{align*}
because $w<1/(n-1)$. So $F_{-1}$ cannot reach $-\mathcal K_{n,-1}$ from above: at a first contact point we would have $F_{-1}'\leq 0$, contradicting \eqref{riccati}.

(c) For $c=1$, $u>0$ by (b). Hence $mw<1/n$, and $w'=1-(n-1)mw>1/n$. (Where $m\leq0$ one even has $w'\geq1$.) For $c=-1$, (b) gives $u=nw^2F/2>-(n-1)w^2$. Therefore $nmw<1+(n-1)w^2$ and $w'>(1-(n-1)^2w^2)/n$. On $(0,r_0]$ we have $w\leq1/(nh)$, which gives the last assertion.

(d) Let $g=F-F_0$ and $\mathfrak q=g/w^2$. Since $c=(n+2)F_0/(2n)$, \eqref{riccati} reads $g'=-(n+2)g/(nw)-(n-2)wF^2/2$. This is equivalent to
\begin{equation}\label{qeq}
w\,\mathfrak q'=-\bar\beta \mathfrak q-\frac{n-2}{2}F^2,\qquad \bar\beta=2w'+\frac{n+2}{n}\geq\frac{n+2}{n},
\end{equation}
where the inequality comes from (c). By Lemma \ref{regularity}, $F-F_0$ and $w^2$ are analytic in $r^2$ with $w^2=r^2/n^2+O(r^4)$. So $\mathfrak q$ is analytic in $r^2$, and $w\mathfrak q'=O(r^2)$. Letting $r\to0$ in \eqref{qeq}, with $w'(0)=1/n$, gives $\mathfrak q(0)=-n(n-2)F_0^2/(2(n+4))\in(-Q,0)$.

If $r_1>0$ were the first zero of $\mathfrak q$, then $\mathfrak q'(r_1)\geq0$. But at $r_1$ we have $F=F_0$, so \eqref{qeq} gives $w\mathfrak q'=-(n-2)F_0^2/2<0$. Hence $\mathfrak q<0$, that is $F<F_0$, for $r>0$. Together with (b) this gives $0<F_1<F_0=\mathcal K_{n,1}$ and $-\mathcal K_{n,-1}<F_{-1}<F_0<0$, hence $|F|\leq \mathcal K_{n,c}$; since $k=F\circ r$ and $k(0)=F_0$, the assertions on $k$ follow. Now, on any interval where $\mathfrak q<-Q$, we get $w\mathfrak q'>Q(n+2)/n-(n-2)\mathcal K_{n,c}^2/2=0$ by $|F|\leq \mathcal K_{n,c}$, so $\mathfrak q$ never drops below $-Q$.

Finally, the bracket in \eqref{riccati} equals $w^2\big[-\frac{n+2}{2n}\mathfrak q-\frac{n-2}{4}F^2\big]$. The two summands have opposite signs, and each has absolute value at most $(n-2)\mathcal K_{n,c}^2/4$.

(e) Since $k=F\circ r$, $s=n^2h^2w^2$ and $ds/dr=2n^2h^2ww'$, \eqref{riccati} gives
\[
s\,k'(s)=\frac{1}{w'}\left[c-\frac{n+2}{2n}F-\frac{n-2}{4}w^2F^2\right].
\]
By (d), (c) and $w\leq1/(nh)$ on $(0,r_0]$,
$|sk'|\leq (n-2)\mathcal K_{n,c}^2w^2\,n/(4\theta)\leq\sigma_{n,c}(h)$.

(f) By the proof of (d), $F-F_0=\mathfrak q w^2$ with $\mathfrak q$ analytic in $r^2$ and $\mathfrak q(0)<0$. Hence $F'=\mathfrak q'w^2+2\mathfrak q ww'=2\mathfrak q(0)ww'+O(r^3)<0$ for small $r>0$, since $w=r/n+O(r^3)$ and $w'\to1/n$. Suppose $r_1$ is the first zero of $F'$. By \eqref{riccati} the bracket vanishes at $r_1$; in particular $F(r_1)\neq0$, since otherwise the bracket would equal $c$. Differentiating \eqref{riccati} at $r_1$ and using $F'(r_1)=0$,
\[
F''(r_1)=\frac{2}{w}\cdot\Big(-\frac{n-2}{2}\,ww'F^2\Big)=-(n-2)\,w'(r_1)F(r_1)^2<0
\]
by (c). But $F'<0$ on $(0,r_1)$ and $F'(r_1)=0$ force $F''(r_1)\geq0$, a contradiction. Hence $F'<0$ on the whole domain. Since $k=F\circ r$ and $dr/ds>0$, also $k'<0$.
\end{proof}

\begin{proposition}\label{modelzero}
The tensor $B$ of \eqref{Bdef} is smooth, symmetric, and trace-free. It agrees with \eqref{old} when $n=2$, and vanishes identically on the rotational reference sphere for every $n\geq2$.
\end{proposition}
\begin{proof}
Both summands in \eqref{Bdef} are trace-free, and Lemma \ref{regularity} supplies smoothness and the $n=2$ statement. On the model, for $s>0$,
\[
A-h\Id=(a-b)\left(\frac{T\otimes T}s-\frac1n\Id\right).
\]
The defining identity $ks=2h(a-b)$ proves $B=0$. Continuity covers the poles.
\end{proof}

\section{The complete equality case}\label{equal-section}
\begin{lemma}\label{local}
Let $n\geq2$ and suppose a CMC immersion in the range \eqref{range} has $B=0$. At every point with $0<|T|^2<1$ its image is locally congruent to an open piece of the reference rotational sphere.
\end{lemma}
\begin{proof}
Put $s=|T|^2$, $t=\sqrt s$, and $e_1=T/t$. The equation $B=0$ gives
\begin{equation}\label{a0}
A=h\Id+\frac{k(s)}{2h}\left(T\otimes T-\frac{s}{n}\Id\right),
\quad Ae_1=a(s)e_1,\quad A|_{e_1^\perp}=b(s)\Id,
\end{equation}
where
\begin{equation}\label{ab}
a(s)=h+\frac{(n-1)k(s)s}{2hn},\qquad
b(s)=h-\frac{k(s)s}{2hn}.
\end{equation}
By \eqref{defk} these coincide with the functions $a,b$ in \eqref{model}. Equations \eqref{compat}--\eqref{grads} imply, for $X\perp e_1$,
\begin{equation}\label{levelid}
X(t)=X(\nu)=0,\quad e_1(t)=\nu a,\quad
e_1(\nu)=-at,\quad\nabla_{e_1}e_1=0.
\end{equation}
Since $T$ is the gradient of the height function, the distribution $e_1^\perp$ is the tangent bundle to horizontal level hypersurfaces. Since $\nu\neq0$, the field $E=(e_1-t\xi)/\nu$ is horizontal and has unit length, because $|e_1-t\xi|^2=1-t^2=\nu^2$. Moreover
\begin{equation}\label{E}
N=-tE+\nu\xi,\qquad e_1=\nu E+t\xi.
\end{equation}
Let $X\perp e_1$. Then $X\perp\xi$, and $X(t)=X(\nu)=0$. Hence $-bX=-AX=\bar\nabla_XN=-t\,\bar\nabla_XE$, which gives
\begin{equation}\label{umb}
\nabla_X^{\QQ_c}E=\frac{b(s)}{t}X
=m(r(s))X.
\end{equation}
Next we show $\bar\nabla_{e_1}E=0$. The Gauss formula and \eqref{levelid} give $\bar\nabla_{e_1}e_1=\ip{Ae_1}{e_1}N=aN$. Differentiating $e_1=\nu E+t\xi$ along $e_1$ and using $e_1\nu=-at$, $e_1t=\nu a$ gives $-atE+\nu\bar\nabla_{e_1}E+\nu a\xi=a(-tE+\nu\xi)$. Hence $\nu\bar\nabla_{e_1}E=0$, and $\nu\neq0$ on the present neighborhood. Since $E$ is horizontal and the metric is a product, $\bar\nabla_VE=\nabla^{\QQ_c}_{\pi_*V}E$, where $\pi$ is the projection to the base. Note also $\pi_*X=X$ for $X\perp e_1$, and $\pi_*e_1=\nu E$.

Put $r=r(s)$ on the neighborhood. Since $t=nhw(r)$ and $a=nhw'(r)>0$, the identity $e_1t=\nu a$ gives $e_1r=\nu$, and $X(r)=0$ for $X\perp e_1$. Realize $\QQ_c^n$ as the quadric $\{\ip{x_0}{x_0}=c\}$ in $\RR^{n+1}$ (Euclidean for $c=1$; Lorentzian, upper sheet, for $c=-1$), with position vector $x_0$. For fields $V,W$ tangent to $\QQ_c^n$, the flat derivative is $D_WV=\nabla^{\QQ_c}_WV-c\ip WV x_0$. Define
\[
p=f_c'(r)\,x_0-f_c(r)\,E .
\]
For $X\perp e_1$, \eqref{umb} gives $D_XE=m(r)X$, so $X(p)=f_c'X-f_cmX=0$. Along $e_1$, we have $D_{e_1}x_0=\nu E$ and $D_{e_1}E=\nabla^{\QQ_c}_{\nu E}E-c\nu x_0=-c\nu x_0$. Therefore
\[
e_1(p)=\nu f_c''x_0+\nu f_c'E-\nu f_c'E+c\nu f_cx_0=\nu\,(f_c''+cf_c)\,x_0=0 .
\]
Thus $p$ is constant near the point. Moreover $\ip pp=c(f_c'^2+cf_c^2)=c$, so $p\in\QQ_c^n$, and $\ip{x_0}{p}=cf_c'(r)$. This says exactly that the base point $x_0$ lies at distance $r=r(s)$ from the fixed point $p$. (For $n=2$ the same computation applies to the level curves.)

For the height $z$, we have $X(z)=0$ and $e_1(z)=\ip{e_1}\xi=t$. Let $Z(r)=\int_0^r\rho/\sqrt{1-\rho^2}$ be the model profile of \eqref{model}. It is finite at $r_0$, since $1-\rho\sim a(1)(r_0-r)$ with $a(1)>0$. Let $\varsigma=\operatorname{sign}\nu$. Since $\rho(r(s))=t$ and $\sqrt{1-t^2}=|\nu|$,
\[
e_1\big(z-\varsigma Z(r)\big)=t-\varsigma\frac{t}{|\nu|}\,\nu=0,\qquad X\big(z-\varsigma Z(r)\big)=0 .
\]
So $z-\varsigma Z(r)$ is constant near the point. Hence a neighborhood is mapped into the set $\{(x_0,z):\ d(x_0,p)=r,\ z=z_c+\varsigma Z(r)\}$. For $\varsigma=1$ this is the lower hemisphere of a reference sphere. For $\varsigma=-1$ it is the upper hemisphere of the reference sphere whose lowest point has height $z_c-2Z(r_0)$. Both sets are hypersurfaces, so the neighborhood is mapped onto an open piece of a reference sphere. Local congruence follows.
\end{proof}

\begin{lemma}\label{analytic}
Let a connected CMC immersion into the real-analytic ambient manifold $\QQ_c^n\times\RR$ agree on a nonempty open subset, up to a fixed ambient isometry, with an open subset of an embedded real-analytic CMC hypersurface $\mathcal S$ which is a closed subset of $\QQ_c^n\times\RR$. Then its entire image lies in that same $\mathcal S$. The reference spheres, being compact and embedded, satisfy this hypothesis.
\end{lemma}
\begin{proof}
\emph{Analyticity.} Let $M=\QQ_c^n\times\RR$, whose metric is real analytic in suitable coordinates. Fix $q\in\Sigma$. In normal coordinates of $M$ centered at $x(q)$, which are analytic, a neighborhood of $q$ is mapped onto the graph of a function $u$ over a ball $W$ of the tangent hyperplane $\mathbf P=dx(T_q\Sigma)$. The CMC condition is a quasilinear elliptic equation for $u$ with real-analytic coefficients. By Morrey's theorem \cite{Morrey}, $u$ is real analytic. The same applies to $\mathcal S$. The reference spheres are, in addition, analytic because their profiles solve analytic ODEs.

\emph{Continuation.} Replace $\mathcal S$ by its image under the fixed isometry. Let $V\subset\Sigma$ be the set of points having a neighborhood mapped into $\mathcal S$. By hypothesis $V$ is open and nonempty. We show that $V$ is closed.

Let $p_j\in V$ with $p_j\to p$. Since $\mathcal S$ is a closed subset of the ambient space, $x(p)\in\mathcal S$. The tangent spaces satisfy $dx(T_{p_j}\Sigma)=T_{x(p_j)}\mathcal S$, and by continuity $dx(T_p\Sigma)=T_{x(p)}\mathcal S=:\mathbf P$.

Choose normal coordinates at $x(p)$, identify $\mathbf P$ with a hyperplane of the coordinate space, and let $\nu_\mathbf P$ be a unit normal to $\mathbf P$. Choose a small ball $W\subset\mathbf P$ centered at $0$ and $\epsilon>0$, and let $Z=\{y+t\nu_\mathbf P:\ y\in W,\ |t|<\epsilon\}$ be the corresponding small coordinate cylinder around $x(p)$. We take $W$ and $\epsilon$ so small that the following hold. A neighborhood $U_p$ of $p$ is mapped diffeomorphically onto the graph in $Z$ of an analytic function $u:W\to(-\epsilon,\epsilon)$. Because $\mathcal S$ is embedded and tangent to $\mathbf P$ at $x(p)$, $\mathcal S\cap Z$ is the graph of an analytic function $v:W\to(-\epsilon,\epsilon)$. Only this local cylinder is used.

For large $j$, $p_j\in U_p$, and a neighborhood of $p_j$ inside $U_p$ is mapped into $\mathcal S\cap Z$, that is, into the graph of $v$. Its projection is an open subset of $W$ on which $u=v$. Since $W$ is connected, the identity theorem gives $u\equiv v$ on $W$. Hence $U_p$ is mapped into $\mathcal S$, and $p\in V$.

Since $\Sigma$ is connected, $V=\Sigma$.
\end{proof}

\begin{proof}[Proof of Theorem \ref{main-equality}]
The first assertions follow from Proposition \ref{modelzero}. Suppose $n\geq2$, $B=0$, and $\Sigma$ connected. The set $U=\{0<|T|^2<1\}$ is nonempty. Otherwise the continuous function $|T|^2$ takes values only in $\{0,1\}$ and is constant. If it is zero, the height function is constant, so the immersion lies in a totally geodesic horizontal slice. Then $A=0$, contradicting $h>0$. If it is one, then $\nu\equiv0$, and \eqref{compat} gives $\ip{AX}{T}=-X(\nu)=0$ for all $X$, so $AT=0$. But \eqref{a0} at $s=1$ gives $AT=a(1)T$ with $a(1)=nh\,w'(r_0)>0$.

Lemma \ref{local} identifies an open subset with a portion of one fixed reference sphere $\mathcal S$. Lemma \ref{analytic} places the entire image in this same sphere, including points at which $T=0$ or $\nu=0$. Since $\mathcal S$ is an embedded submanifold, $x$ factors through a smooth immersion $\Sigma\to\mathcal S$ between manifolds of the same dimension. It is therefore a local isometry for the induced metrics. A local isometry from a complete connected Riemannian manifold to a connected one is a covering map \cite[Ch.~7]{dC}, so $\Sigma\to\mathcal S$ is a surjective covering. Since $\mathcal S\cong\SSS^n$ is simply connected for $n\geq2$, the covering has one sheet.
\end{proof}

\section{The linear Codazzi defect and a Simons identity}\label{simons-section}
\begin{proposition}\label{defect}
For $B$ as in \eqref{Bdef}, $s=|T|^2$, $k=k_{n,c,h}(s)$, and $D(X,Y)=(\nabla_XB)Y-(\nabla_YB)X$, one has
\begin{align}\label{linearD}
D(X,Y)=\frac{\nu}{2h}\Big\{&
-k\big(\ip YT BX-\ip XT BY\big)\notag\\
&+\frac{2(k+sk')}{n}\big(\ip{BX}{T}Y-\ip{BY}{T}X\big)\notag\\
&-2k'\big(\ip{BX}{T}\ip YT-\ip{BY}{T}\ip XT\big)T\Big\}.
\end{align}
In particular $D$ vanishes whenever $B$ does.
\end{proposition}
\begin{proof}
For $P=T\otimes T$ one has
$(\nabla_XP)Y=\nu\ip{AX}{Y}T+\nu\ip YT AX$.
The first term is symmetric in $X,Y$. The product rule applied to $k(s)(P-s\Id/n)$, together with \eqref{codA} and \eqref{grads}, gives
\begin{align}\label{uncancelled}
D(X,Y)=\nu\Big\{&
-2hc(\ip XT Y-\ip YT X)
-k(\ip YT AX-\ip XT AY)\notag\\
&+\frac{2(k+sk')}{n}(\ip{AX}{T}Y-\ip{AY}{T}X)\notag\\
&-2k'(\ip{AX}{T}\ip YT-\ip{AY}{T}\ip XT)T\Big\}.
\end{align}
Substitute $A=A_0+B/(2h)$, where $A_0$ is the right side of \eqref{a0}. In a basis $e_1=T/\sqrt s$, $e_i\perp T$, the expression involving only $A_0$ vanishes on $(e_i,e_j)$, $i,j>1$, and on $(e_1,e_i)$ equals
\[
\nu\sqrt s\left[-2hc+kb+\frac{2a}{n}(k+sk')\right]e_i.
\]
This bracket is zero by the scalar identity
\begin{equation}\label{kODE}
\frac{a}{hn}(k+sk')+\frac{kb}{2h}=c.
\end{equation}
Here is a direct independent verification. From \eqref{model} and $s=(nhw)^2$,
\[
2a\frac{db}{ds}-\frac{(a-b)b}{s}
=\frac{w'm+wm'}{w}-\frac{(w'-wm)m}{w}
=m'+m^2=-c.
\]
Substituting \eqref{ab} into the left side and rearranging gives \eqref{kODE}. Thus the $A_0$ contribution vanishes. The remaining contribution of $B/(2h)$ to \eqref{uncancelled} is precisely \eqref{linearD}. Smoothness covers $s=0$.
\end{proof}

The linearity of the defect characterizes the coefficient.

\begin{proposition}\label{canonical}
Let $\tilde k\in C^0([0,1])\cap C^1((0,1])$, and let $\tilde B$ be defined by \eqref{Bdef} with $\tilde k$ in place of $k_{n,c,h}$. Let $\tilde a,\tilde b$ be given by \eqref{ab} with $\tilde k$, and put
\[
\Xi(s)=-2hc+\tilde k\tilde b+\frac{2\tilde a}{n}(\tilde k+s\tilde k'),\qquad 0<s\leq1.
\]
Then, on every CMC immersion with mean curvature $h$ and wherever $T\neq0$,
\begin{equation}\label{Dgeneral}
\tilde D(X,Y)=\nu\,\Xi(s)\big(\ip XT Y-\ip YT X\big)+\tilde L(\tilde B)(X,Y),
\end{equation}
where $\tilde L(\tilde B)$ denotes the right side of \eqref{linearD} with $(k,B)$ replaced by $(\tilde k,\tilde B)$. Moreover $\Xi\equiv0$ on $(0,1]$ if and only if $\tilde k=k_{n,c,h}$.
\end{proposition}
The $C^1$ regularity on $(0,1]$ is needed to define the defect $\tilde D$; continuity at $s=0$ is used to determine $\tilde k(0)$. The proposition is a uniqueness statement within the class of tensors of the form \eqref{Bdef} with such a coefficient.
\begin{proof}
The derivation of \eqref{uncancelled} and the splitting $A=\tilde A_0+\tilde B/(2h)$ use only that $\tilde k$ is $C^1$. The $\tilde A_0$-part is an antisymmetric vector-valued form. It vanishes on $(e_i,e_j)$ with $i,j>1$ and equals $\nu\sqrt s\,\Xi(s)e_i$ on $(e_1,e_i)$. This is exactly the first term in \eqref{Dgeneral}.

By Proposition \ref{defect}, $\Xi\equiv0$ for $\tilde k=k_{n,c,h}$. Conversely, suppose $\Xi\equiv0$. Since $\tilde k$ is continuous at $0$, $\tilde a(s)=h+\frac{(n-1)s\tilde k(s)}{2hn}\to h>0$, so $\tilde a>0$ on some interval $(0,s_0]$. There we may divide by $\tilde a$, and $\Xi\equiv0$ becomes
\[
s\tilde k'=\Psi(s,\tilde k),\qquad
\Psi(s,\kappa)=\frac{n\,(2hc-\tilde b\kappa)}{2\tilde a}-\kappa,
\]
where $\tilde a=h+\frac{(n-1)s\kappa}{2hn}$ and $\tilde b=h-\frac{s\kappa}{2hn}$.
Here $\Psi$ is smooth near $\{s=0\}$, and $\Psi(0,\kappa)=nc-(n+2)\kappa/2$. If $\tilde k(0)\neq F_0$, then $s\tilde k'(s)$ tends to a nonzero limit, so $\tilde k$ is unbounded as $s\to0$. This is a contradiction. Hence $\tilde k(0)=F_0=k(0)$.

Let $v=\tilde k-k$. Since $\partial_\kappa\Psi(0,F_0)=-(n+2)/2$, there is $s_1>0$ such that $s\,(v^2)'\leq-\frac{n+2}{2}v^2$ on $(0,s_1]$. So $s^{(n+2)/2}v(s)^2$ is nonincreasing there. It tends to $0$ as $s\to0$, which forces $v\equiv0$ on $(0,s_1]$.

On $[s_1,1]$ we argue near the graph of the known solution $k$. At a point $s_*$ with $\tilde k(s_*)=k(s_*)$ we have $\tilde a(s_*)=a(s_*)=nhw'(r(s_*))>0$, so $\tilde a>0$ nearby. There $\Xi\equiv0$ is again equivalent to $\tilde k'=\Psi(s,\tilde k)/s$, whose right side is locally Lipschitz; the same holds for $k$. By uniqueness, $v=0$ near $s_*$. The set where $v=0$ is therefore open and closed in $[s_1,1]$, and nonempty, so $v\equiv0$. No sign condition on $\tilde a$ is assumed away from the graph of $k$.
\end{proof}

\begin{lemma}\label{weitzen}
Let $B$ be any symmetric trace-free tensor on a Riemannian $n$-manifold and put $D_{ijk}=\nabla_iB_{jk}-\nabla_jB_{ik}$. In a local orthonormal frame diagonalizing $B$, with eigenvalues $\beta_i$,
\begin{align}\label{pointwise}
\frac12\Delta|B|^2={}&|\nabla B|^2+
\sum_{i<j}K_{ij}(\beta_i-\beta_j)^2
+B_{jk}\nabla_iD_{ijk}
+B_{jk}\nabla_j(\diver B)_k .
\end{align}
If the manifold is closed, integration yields \eqref{main-int}.
\end{lemma}
\begin{proof}
We use index notation with respect to a local orthonormal frame and write $\nabla^2_{ij}B_{kl}=(\nabla^2B)(e_i,e_j;e_k,e_l)$, where $\nabla^2_{X,Y}=\nabla_X\nabla_Y-\nabla_{\nabla_XY}$. With our convention, $\nabla^2_{X,Y}-\nabla^2_{Y,X}=R(X,Y)$ acting as a derivation. Hence the Ricci identity reads
\[
\nabla^2_{ij}B_{kl}-\nabla^2_{ji}B_{kl}=-B\big(R(e_i,e_j)e_k,e_l\big)-B\big(e_k,R(e_i,e_j)e_l\big).
\]
By definition of $D$ we have $\nabla_iB_{jk}=\nabla_jB_{ik}+D_{ijk}$. Therefore
\begin{align*}
\Delta B_{jk}&=\sum_i\nabla^2_{ii}B_{jk}=\sum_i\nabla_iD_{ijk}+\sum_i\nabla^2_{ij}B_{ik}\\
&=\sum_i\nabla_iD_{ijk}+\nabla_j(\diver B)_k+\sum_i\big(\nabla^2_{ij}-\nabla^2_{ji}\big)B_{ik}.
\end{align*}
Here we used that contraction commutes with $\nabla$, so $\sum_i\nabla^2_{ji}B_{ik}=\nabla_j(\diver B)_k$.

Now take the frame to diagonalize $B$ at the point, $Be_a=\beta_ae_a$, and set $K_{ii}=0$. Using the Ricci identity and $\ip{R(e_i,e_j)e_i}{e_j}=-K_{ij}$,
\begin{align*}
\sum_{i,j,k}B_{jk}\big(\nabla^2_{ij}-\nabla^2_{ji}\big)B_{ik}
&=\sum_{i,j}\beta_j^2K_{ij}-\sum_{i,j}\beta_i\beta_jK_{ij}\\
&=\sum_{i,j}K_{ij}\beta_j(\beta_j-\beta_i)=\sum_{i<j}K_{ij}(\beta_i-\beta_j)^2 .
\end{align*}
Since $\frac12\Delta|B|^2=|\nabla B|^2+\sum_{j,k}B_{jk}\Delta B_{jk}$, this proves \eqref{pointwise}.

On a closed manifold, apply the divergence theorem to the vector fields $B_{jk}D_{ijk}$ and $B_{jk}(\diver B)_k$ (index $i$, resp.\ $j$, free). This gives
\[
\int B_{jk}\nabla_iD_{ijk}=-\int(\nabla_iB_{jk})D_{ijk}=-\frac12\int\sum_{i,j,k}D_{ijk}^2=-\frac12\int|D|^2,
\]
where we used $D_{ijk}=-D_{jik}$, so that $(\nabla_iB_{jk})D_{ijk}=\frac12(\nabla_iB_{jk}-\nabla_jB_{ik})D_{ijk}$. By the symmetry of $B$, $\sum_j\nabla_jB_{jk}=(\diver B)_k$, and so
\[
\int B_{jk}\nabla_j(\diver B)_k=-\int|\diver B|^2 .
\]
Integrating \eqref{pointwise} gives \eqref{main-int}.
\end{proof}
\begin{proof}[Proof of Theorem \ref{main-identity}]
Apply Lemma \ref{weitzen} to the trace-free tensor of Proposition \ref{modelzero}. Proposition \ref{defect} supplies its explicit defect.
\end{proof}

\section{Explicit estimates and the gap}\label{gap-section}
Throughout this section $n\geq3$, $h$ lies in \eqref{range}, $\mathcal K=\mathcal K_{n,c}$ and $\sigma=\sigma_{n,c}(h)$ are as in \eqref{Kn}--\eqref{sigma}, and $c_-=\max\{0,-c\}$. For the reader's convenience we collect the notation.

\medskip
\begin{center}
\small
\begin{tabular}{ll}
$k=k_{n,c,h}(s)$, $F_c$, $F_0=\frac{2nc}{n+2}$ & model coefficient, \eqref{defk}--\eqref{kzero}\\
$\mathcal K=\mathcal K_{n,c}$, $\theta=\theta_{n,c}(h)$ & \eqref{Kn}; $|k|\leq \mathcal K$ by Lemma \ref{bounds}(d)\\
$\sigma=\sigma_{n,c}(h)$ & \eqref{sigma}; $|sk'|\leq\sigma$\\
$Q=\frac{n(n-2)\mathcal K^2}{2(n+2)}$ & Lemma \ref{bounds}(d)\\
$\delta_D,\ \delta_{\diver},\ \varepsilon,\ a_0,\ a_1$ & \eqref{constants}\\
$x_+$ & \eqref{xplus}\\
$\zeta$, $\Lambda$, $q_{n,c,h}$, $\mathcal P_{n,c,h}$, $R_*$, $X_{\rm conv}$ & Section \ref{sharp-section}\\
$I_{n,c}(h)$, $G$, $\Theta$, $\Pi$, $x_{\rm opt}$ & Section \ref{opt-section}
\end{tabular}
\end{center}
\medskip

Define
\begin{equation}\label{constants}
\begin{aligned}
\delta_D&=\frac{\sqrt2}{4}\left[\mathcal K+\frac{2\sqrt{n-1}}{n}(\mathcal K+\sigma)+2\sigma\right],\\
\delta_{\diver}&=\frac14\left[\frac{(n-2)\mathcal K+2\sigma}{n}+2\sigma\right],\qquad
\varepsilon=\frac{\delta_D^2}{2}+\delta_{\diver}^2,\\
a_0&=nh^2-nc_--\frac1{h^2}\left(\frac{(n-1)\mathcal K^2}{4n}+\varepsilon\right),\\
a_1&=\frac{(n-2)\sqrt n}{2\sqrt{n-1}}+\frac{\mathcal K}{2h^2}\sqrt{\frac{n-1}{n}}.
\end{aligned}
\end{equation}

\begin{lemma}\label{error}
The divergence of $B$ is
\begin{equation}\label{divB}
\diver B=\frac{\nu}{2h}\left[\frac{(2-n)k+2sk'}{n}\,BT-2k'\ip{BT}{T}\,T\right],
\end{equation}
and
\begin{equation}\label{error1}
|D|\leq\frac{\delta_D}{h}|B|,\qquad |\diver B|\leq\frac{\delta_{\diver}}{h}|B|,\qquad
\frac12|D|^2+|\diver B|^2\leq\frac{\varepsilon}{h^2}|B|^2.
\end{equation}
\end{lemma}
\begin{proof}
Because $\tr B=0$, $(\diver B)_j=\sum_iD_{iji}$. Contracting \eqref{linearD} term by term gives $k\,BT$, $-\frac{2(n-1)}{n}(k+sk')BT$, and $2sk'BT-2k'\ip{BT}{T}T$. Adding these gives \eqref{divB}. For $n=2$, where $k\equiv c$ and $k'\equiv0$, it vanishes, as it should.

For the norms, let $\{e_i\}$ be orthonormal. Then
\begin{gather*}
\sum_{i,j}\big|\ip{e_j}{T}Be_i-\ip{e_i}{T}Be_j\big|^2=2|T|^2|B|^2-2|BT|^2,\\
\sum_{i,j}\big|(BT)_ie_j-(BT)_je_i\big|^2=2(n-1)|BT|^2,
\end{gather*}
and $\sum_{i,j}\big((BT)_iT_j-(BT)_jT_i\big)^2\leq2|BT|^2|T|^2$. With $|BT|\leq|B||T|$ this gives
\[
|D|\leq\frac{|\nu||T|}{2h}|B|\left[\sqrt2|k|+\frac{2\sqrt{2(n-1)}}{n}|k+sk'|+2\sqrt2|sk'|\right].
\]
Similarly, \eqref{divB} gives $|\diver B|\leq\frac{|\nu||T|}{2h}|B|\big[|(n-2)k-2sk'|/n+2|sk'|\big]$. Now use $|\nu||T|\leq\frac12$, $|k|\leq \mathcal K$ and $|sk'|\leq\sigma$ (Lemma \ref{bounds}).
\end{proof}

\begin{lemma}\label{curvterm}
At each point, in an orthonormal frame diagonalizing $B$, with $T_i=\ip T{e_i}$,
\begin{equation}\label{curvexact}
\begin{aligned}
\sum_{i<j}K_{ij}(\beta_i-\beta_j)^2={}&nh^2|B|^2+\frac n2\sum_i\beta_i^3-\frac{|B|^4}{4h^2}\\
&+\Big(c(n-s)-\frac{ks}2\Big)|B|^2+n\Big(\frac k2-c\Big)|BT|^2\\
&-\frac{k}{2h^2}|B|^2\ip{BT}{T}-\frac{k^2s}{4h^2}|BT|^2.
\end{aligned}
\end{equation}
Therefore
\begin{equation}\label{curvlower}
\sum_{i<j}K_{ij}(\beta_i-\beta_j)^2\geq|B|^2\left[nh^2-nc_--\frac{(n-1)\mathcal K^2}{4nh^2}-a_1|B|-\frac{|B|^2}{4h^2}\right].
\end{equation}
\end{lemma}
\begin{proof}
Write $\gamma_i=\beta_i/(2h)$, $\kappa=k/(2h)$, and $A=\alpha+E$, where $\alpha=\operatorname{diag}(h+\gamma_i)$ and $E=\kappa(T\otimes T-\frac sn\Id)$. For $i\neq j$ we have $A_{ij}=E_{ij}=\kappa T_iT_j$. By \eqref{gauss},
\[
K_{ij}=c(1-T_i^2-T_j^2)+\alpha_i\alpha_j+\alpha_iE_{jj}+\alpha_jE_{ii}+E_{ii}E_{jj}-E_{ij}^2.
\]
Multiply by $(\beta_i-\beta_j)^2$ and sum over $i<j$. All sums reduce to the elementary identity
\begin{equation}\label{star}
\sum_{i,j}u_iv_j(\beta_i-\beta_j)^2=U_2V_0-2U_1V_1+U_0V_2,
\end{equation}
where $U_p=\sum_iu_i\beta_i^p$ and $V_p=\sum_jv_j\beta_j^p$. It is applied with $u,v\in\{\mathbf 1,(\gamma_i),(T_i^2),(E_{ii})\}$. Here $\sum_i\beta_i=0$, $\sum\beta_i^2=|B|^2$, $\sum T_i^2=s$, $\sum T_i^2\beta_i=\ip{BT}T$, $\sum T_i^2\beta_i^2=|BT|^2$ and $\sum_iE_{ii}=0$. For instance, $\sum_{i<j}(\beta_i-\beta_j)^2=n|B|^2$, $\sum_{i<j}(T_i^2+T_j^2)(\beta_i-\beta_j)^2=n|BT|^2+s|B|^2$, $\sum_{i<j}(\beta_i+\beta_j)(\beta_i-\beta_j)^2=n\sum\beta_i^3$ and $\sum_{i<j}\beta_i\beta_j(\beta_i-\beta_j)^2=-|B|^4$. One obtains:
\begin{enumerate}[label=(\alph*)]
\item the $c$-terms give $c[(n-s)|B|^2-n|BT|^2]$;
\item $\alpha_i\alpha_j=h^2+h(\gamma_i+\gamma_j)+\gamma_i\gamma_j$ gives $nh^2|B|^2+\frac n2\sum\beta_i^3-|B|^4/(4h^2)$, as in \cite{AdC};
\item the part $h(E_{ii}+E_{jj})$ gives $\frac k2(n|BT|^2-s|B|^2)$;
\item the part $\gamma_iE_{jj}+\gamma_jE_{ii}$ gives $-\frac{k}{2h^2}|B|^2\ip{BT}{T}$, since $\sum_j E_{jj}=0$ and $\sum_i\gamma_i=0$;
\item $E_{ii}E_{jj}-E_{ij}^2=\kappa^2[s^2/n^2-(s/n)(T_i^2+T_j^2)]$ gives $-\kappa^2s|BT|^2$.
\end{enumerate}
This proves \eqref{curvexact}.

For \eqref{curvlower}, use that a trace-free $B$ satisfies $\max_i\beta_i^2\leq\frac{n-1}{n}|B|^2$. Hence $|BT|^2\leq\frac{n-1}{n}s|B|^2$ and $|\ip{BT}T|\leq s\sqrt{\frac{n-1}n}|B|$. Okumura's inequality \cite{Okumura} gives $\sum\beta_i^3\geq-\frac{n-2}{\sqrt{n(n-1)}}|B|^3$.

Now consider the two terms in the second group of \eqref{curvexact}. If $c=1$, then $0\leq k<2$, so $\frac k2-1<0$, and they are bounded below by
\[
\Big(n-s-\frac{ks}{2}+(n-1)\big(\tfrac k2-1\big)s\Big)|B|^2=\Big(n(1-s)+\frac{(n-2)ks}{2}\Big)|B|^2\geq0.
\]
If $c=-1$, then $\frac k2+1>0$ because $|k|<2$, and they are bounded below by $-(n-s)|B|^2\geq-n|B|^2$. The remaining terms are estimated with $|k|\leq \mathcal K$ and $s\leq1$.
\end{proof}

\begin{proof}[Proof of Theorem \ref{main-gap}]
By \eqref{main-int}, \eqref{error1} and \eqref{curvlower},
\[
0\geq\int_\Sigma\Big\{|\nabla B|^2+|B|^2\,P(|B|)\Big\},\qquad
P(x)=a_0-a_1x-\frac{x^2}{4h^2}.
\]
Since $a_0>0$, the polynomial $P$ is nonnegative on $[0,x_+]$, where $x_+$ is given by \eqref{xplus}. To allow equality in \eqref{pinch}, we use the slack in \eqref{curvlower}. In its proof the term $-\frac{k^2s}{4h^2}|BT|^2$ was bounded below by $-\frac{(n-1)k^2s^2}{4nh^2}|B|^2$, and then $k^2s^2$ was replaced by $\mathcal K^2$. Keeping the difference, we get
\[
0\geq\int_\Sigma\Big\{|\nabla B|^2+|B|^2\big(P(|B|)+\mathfrak e\big)\Big\},\qquad
\mathfrak e=\frac{(n-1)(\mathcal K^2-k^2s^2)}{4nh^2}.
\]
The function $\mathfrak e$ is continuous and strictly positive on $\Sigma$. For $c=-1$, $|k|<\mathcal K$ everywhere by Lemma \ref{bounds}(d). For $c=1$, $0<k<F_0=\mathcal K$ when $s>0$ by Lemma \ref{bounds}(d), and $k^2s^2=0<\mathcal K^2$ when $s=0$. Under \eqref{pinch}, $P(|B|)\geq0$, hence $\int_\Sigma\mathfrak e|B|^2\leq0$. Thus $B\equiv0$, and Theorem \ref{main-equality} concludes.

For the limit: $\sigma=O(h^{-2})$ and $\varepsilon=O(1)$, so $a_0/h^2\to n$ and $a_1\to a_1^\infty=\frac{(n-2)\sqrt n}{2\sqrt{n-1}}$. Since $(a_1^\infty)^2+n=\frac{n^3}{4(n-1)}$, we get $x_+/h^2\to2\big(\frac{n\sqrt n}{2\sqrt{n-1}}-a_1^\infty\big)=2\sqrt{n/(n-1)}$.

The claim $a_0>0$ for $h^2\geq8$ is Lemma \ref{a0-positive}, and the assertion concerning $x_{\rm opt}$ is Theorem \ref{opt-gap}.
\end{proof}

\begin{lemma}\label{a0-positive}
Let $n\geq3$, $c=\pm1$ and $h^2\geq8$. Then $\sigma<1/7$, $\varepsilon<3/2$, $a_1<\frac{n-1}2+\frac18$, and
\[
a_0>nh^2-n-\frac5{16}>0 .
\]
\end{lemma}
\begin{proof}
Since $\mathcal K<2$, we have $\sigma<(n-2)/(nh^2\theta)<1/(8\theta)$; moreover $\theta=1$ for $c=1$ and $\theta\geq1-1/h^2\geq7/8$ for $c=-1$. Hence $\sigma<1/7$. Since $2\sqrt{n-1}/n\leq2\sqrt2/3$ for $n\geq3$, this gives
\[
\delta_D\leq\frac{\sqrt2}{4}\Big[2+\frac{2\sqrt2}{3}\cdot\frac{15}{7}+\frac27\Big]=\frac{4\sqrt2+5}{7},\qquad
\delta_{\diver}\leq\frac14\Big[2+\frac27\Big]=\frac47 .
\]
Here we used $\frac{(n-2)\mathcal K+2\sigma}{n}<\frac{2(n-2)}{n}+\frac{2}{7n}<2$. Hence
\[
\varepsilon\leq\frac{(4\sqrt2+5)^2}{98}+\frac{16}{49}=\frac{89+40\sqrt2}{98}<\frac32,
\]
where the last inequality is equivalent to $40\sqrt2<58$, that is, $3200<3364$. Since $\frac{(n-1)\mathcal K^2}{4n}<1$ and $h^2\geq8$,
\[
a_0>nh^2-n-\frac{1+3/2}{h^2}\geq nh^2-n-\frac{5}{16}\geq7n-\frac5{16}>0.
\]
Also $(n-2)\sqrt n\leq(n-1)^{3/2}$, which is equivalent to $n^2-n-1\geq0$. So $a_1<\frac{n-1}2+\frac18$.
\end{proof}

\begin{proof}[Proof of Corollary \ref{complete}]
Let $h^2\geq8$. First, $|T\otimes T-\frac sn\Id|=s\sqrt{(n-1)/n}\leq1$ and $|k|<2$. Under $|B|\leq h^2/4$, \eqref{Bdef} gives
\[
|A-h\Id|\leq\frac{|B|}{2h}+\frac{|k|}{2h}\leq\frac h8+\frac1h\leq\frac h4.
\]
So all principal curvatures are at least $3h/4$. By \eqref{gauss}, every sectional curvature satisfies $K_\Sigma\geq\frac{9h^2}{16}-1\geq\frac{h^2}{4}$. By the Bonnet--Myers theorem, $\Sigma$ is compact.

It remains to check that $P(h^2/4)>0$, so that $|B|\leq h^2/4<x_+$ and Theorem \ref{main-gap} applies. By Lemma \ref{a0-positive},
\[
P(h^2/4)=a_0-\frac{a_1h^2}{4}-\frac{h^2}{64}
>h^2\Big(\frac{7n+1}{8}-\frac3{64}\Big)-n-\frac5{16}\geq6n+\frac5{16}>0.
\]
\end{proof}

The pinching of Theorem \ref{main-gap} forces strict convexity. Put
\[
\eta_0=\sqrt{\tfrac{n-1}{n}},\qquad \varkappa_1=\frac{\mathcal K}{2hn},\qquad \varkappa_{-1}=\frac{(n-1)\mathcal K}{2hn}.
\]

\begin{proposition}\label{convex-forced}
Let $n\geq3$ and let $h$ be in \eqref{range} with $a_0>0$. At a point where $|B|\leq x_+$, every principal curvature satisfies
\[
\lambda_{\min}(A)\geq h-\frac{\eta_0|B|}{2h}-\varkappa_c>0 .
\]
In particular, under the hypotheses of Theorem \ref{main-gap}, $\Sigma$ is strictly convex.
\end{proposition}
\begin{proof}
By \eqref{Bdef}, $A=h\Id+\frac1{2h}B+E$ with $E=\frac{k}{2h}(T\otimes T-\frac sn\Id)$. Since $B$ is trace-free, $\lambda_{\min}(B)\geq-\eta_0|B|$. The eigenvalues of $E$ are $\frac{k}{2h}\frac{(n-1)s}{n}$ (on $T$) and $-\frac{k}{2h}\frac sn$ (on $T^\perp$). By Lemma \ref{bounds}(d), $0\leq k\leq \mathcal K$ for $c=1$ and $-\mathcal K\leq k<0$ for $c=-1$, and $s\leq1$. Hence $\lambda_{\min}(E)\geq-\varkappa_c$, which gives the first inequality.

It remains to show $x_+<X_c:=2h(h-\varkappa_c)/\eta_0$. First, $h>\varkappa_c$. For $c=-1$ this is $h^2>(n-1)^2/n^2$, which is \eqref{range}. For $c=1$ it is $h^2>1/(n+2)$; this follows from $a_0>0$, which gives $nh^4>\frac{(n-1)\mathcal K^2}{4n}=\frac{n(n-1)}{(n+2)^2}$, hence $h^2>\frac{\sqrt{n-1}}{n+2}\geq\frac1{n+2}$.

Next, $P(x)=a_0-a_1x-x^2/(4h^2)$ is strictly decreasing on $[0,\infty)$ with $P(0)=a_0>0$, and $x_+$ is its positive zero. So it suffices to prove $P(X_c)<0$. Write $u=h-\varkappa_c>0$. Using $\frac{(n-2)\sqrt n}{2\sqrt{n-1}}\cdot\frac1{\eta_0}=\frac{n(n-2)}{2(n-1)}$, we get
\[
a_1X_c=u\Big(\frac{n(n-2)}{n-1}h+\frac{\mathcal K}h\Big),\qquad \frac{X_c^2}{4h^2}=\frac{n}{n-1}u^2 .
\]
Substituting $h=u+\varkappa_c$ in $nh^2$, the terms in $u^2$ cancel, because $n-\frac{n(n-2)}{n-1}-\frac{n}{n-1}=0$. What remains is
\[
P(X_c)=\frac{n^2\varkappa_c}{n-1}\,u+n\varkappa_c^2-\frac{\mathcal K}{h}\,u-nc_--\frac{1}{h^2}\Big(\frac{(n-1)\mathcal K^2}{4n}+\varepsilon\Big).
\]
For $c=1$, $\varkappa_1=\frac{\mathcal K}{2hn}$, and this equals
\[
P(X_1)=-\frac{(n-2)\mathcal K}{2(n-1)h}\,u-\frac{(n-2)\mathcal K^2}{4nh^2}-\frac{\varepsilon}{h^2}<0 .
\]
For $c=-1$, $\varkappa_{-1}=\frac{(n-1)\mathcal K}{2hn}$ and $c_-=1$. Using $\frac{u}{h}=1-\frac{\varkappa_{-1}}{h}$ and $\mathcal K=\frac{2(n-1)}{n}$, we obtain
\[
P(X_{-1})=\frac{(n-2)\mathcal K}{2}-n-\frac{\varepsilon}{h^2}=-\frac{3n-2}{n}-\frac{\varepsilon}{h^2}<0 .
\]
\end{proof}

\begin{proof}[Proof of Corollary \ref{complete-sphere}]
Put $\delta=h-\frac{\eta_0x_+}{2h}-\varkappa_1$. The proof of Proposition \ref{convex-forced} uses only the pointwise inequality $|B|\leq x_+$, not compactness. It gives $\delta>0$ and $\ip{AX}{X}\geq\delta|X|^2$ for every tangent vector $X$. Let $X,Y$ be orthonormal tangent vectors. Since $\ip XT^2+\ip YT^2\leq|T|^2\leq1$ by Bessel's inequality, the ambient term in \eqref{gauss} with $c=1$ is nonnegative. The restriction of $\ip{A\cdot}{\cdot}$ to the plane spanned by $X,Y$ has both eigenvalues at least $\delta$, so its determinant satisfies $\ip{AX}{X}\ip{AY}{Y}-\ip{AX}{Y}^2\geq\delta^2$. Hence $K_\Sigma\geq\delta^2$ and $\operatorname{Ric}_\Sigma\geq(n-1)\delta^2$. By the Bonnet--Myers theorem, the complete manifold $\Sigma$ is compact, and Theorem \ref{main-gap} applies.
\end{proof}

\begin{remark}\label{convexity}
As noted in the Introduction, Proposition \ref{convex-forced} places $\Sigma$ in the class covered by \cite[Theorem 17]{LMS}. In $\HHH^n\times\RR$ there is also a synthetic route: $\HHH^n\times\RR$ is a Hadamard manifold, so embeddedness follows from \cite{Alexander}, and rotational symmetry from Alexandrov reflection \cite{HH}. Theorem \ref{main-gap} gives a different proof, through the adapted tensor and the integral identity \eqref{main-int}, with an explicit admissible pinching.

For large $h$ the threshold $x_+\approx2\sqrt{n/(n-1)}\,h^2$ corresponds to $|A-h\Id|\lesssim\sqrt{n/(n-1)}\,h$. This is the level at which the extremal configuration of Okumura's inequality acquires a zero principal curvature, and also the Euclidean ($c=0$) case of the bound of \cite{AdC}. The improved estimate of Section \ref{opt-section} shifts this threshold by the bounded amount $\mu_c$ (Theorem \ref{main-asymp}). For $c=1$ the shift is positive, and the improved pinching is no longer controlled by the pointwise convexity estimate of Proposition \ref{convex-forced} (Proposition \ref{nonconvex}).

For $c=1$, completeness suffices under \eqref{pinch}, since compactness follows from the convexity estimate (Corollary \ref{complete-sphere}). For $c=-1$ we obtain compactness only under the stronger hypothesis of Corollary \ref{complete}. For general complete immersions in $\HHH^n\times\RR$ under \eqref{pinch}, the boundary terms in the integration by parts require separate control; the pointwise formula \eqref{pointwise} retains derivatives of $D$, so a maximum-principle gap theorem does not follow from this proof. We do not claim such an extension. For $h$ near the lower end of \eqref{range}, the constant $a_0$ may be negative; Theorem \ref{main-sharp} covers the whole range \eqref{gap-range}.
\end{remark}

\section{A model-adapted estimate on the positivity range}\label{sharp-section}
Throughout this section $n\geq3$ and $h$ lies in \eqref{range}. In Section \ref{gap-section} the coefficient $k=k_{n,c,h}$ entered only through the uniform bounds $|k|\leq \mathcal K$ and $|sk'|\leq\sigma$. Here we use its sign structure and the geometry of the reference sphere. All expressions in the model parameter $r$ use $0\leq r\leq r_0(h)$, with $s=(nhw(r))^2$; endpoint values are taken by smooth continuation. As before, $a(s)=nhw'(r(s))>0$ and $b(s)=nhw(r(s))m(r(s))$ are given by \eqref{ab}.

\begin{lemma}\label{model-signs}
The function $k=k_{n,c,h}$ satisfies $k'(s)<0$ on $[0,1]$. Moreover, put
\begin{equation}\label{zetadef}
\zeta(s)=\frac{k(s)b(s)}{2h}-c
=\frac{m(r)}{w(r)}-nm(r)^2-c,\qquad r=r(s).
\end{equation}
For $0<r<\pi/2$ one has $\zeta<0$ if $c=1$; for every $r>0$ one has $\zeta>0$ if $c=-1$.
\end{lemma}
\begin{proof}
The second expression for $\zeta$ follows from \eqref{defk} and $b=nhwm$. Multiplying \eqref{kODE} by $hn$ and substituting \eqref{ab}, we obtain
\begin{equation}\label{kriccati}
a(s)\,s\,k'(s)=hnc-\frac{h(n+2)}2\, k(s)-\frac{n-2}{4hn}\,s\, k(s)^2 .
\end{equation}
(Equivalently, \eqref{kriccati} is the formula for $sk'$ in the proof of Lemma \ref{bounds}(e), multiplied by $a=nhw'$.) The value $k(0)=F_0=2nc/(n+2)$ cancels the constant terms. Since $a(s)=h+O(s)$ by \eqref{ab}, comparing the coefficients of $s$ in \eqref{kriccati} gives $hk'(0)=-\frac{h(n+2)}{2}k'(0)-\frac{(n-2)F_0^2}{4hn}$, that is,
\begin{equation}\label{kprimezero}
k'(0)=-\frac{2n(n-2)}{h^2(n+2)^2(n+4)}<0.
\end{equation}
Together with Lemma \ref{bounds}(f), which gives $k'<0$ on $(0,1]$, this proves $k'<0$ on $[0,1]$.

For the sign of $\zeta$, set $v=m/(nm^2+c)$. The denominator is positive on the ranges in the statement: trivially for $c=1$, and because $m=\coth r>1$ for $c=-1$. Using $m'=-c-m^2$ one checks that
\begin{equation}\label{comparison-v}
v'+(n-1)mv-1=-\frac{2(cm^2+1)}{(nm^2+c)^2}.
\end{equation}
Put $\phi=w-v$. By \eqref{odew}, $\phi'+(n-1)m\phi=2(cm^2+1)/(nm^2+c)^2$, that is, $(f_c^{n-1}\phi)'=2f_c^{n-1}(cm^2+1)/(nm^2+c)^2$. Both $w$ and $v$ are $O(r)$ as $r\downarrow0$, so $f_c^{n-1}\phi\to0$. For $c=1$ and $0<r<\pi/2$ the right-hand side is positive; hence $\phi>0$, that is, $w>v$, and since $m>0$ this gives $m/w<nm^2+1$, i.e., $\zeta<0$. For $c=-1$ we have $cm^2+1=1-m^2<0$; hence $w<v$, which gives $m/w>nm^2-1$, i.e., $\zeta>0$.

In particular, put $d_+=1-nmw$ when $c=1$ and $d_-=nmw-1$ when $c=-1$. Using \eqref{odew},
\begin{equation}\label{dprime}
\begin{gathered}
d_+'=-nw\,\zeta>0\quad(c=1,\ 0<r<\pi/2),\\
d_-'=nw\,\zeta>0\quad(c=-1).
\end{gathered}
\end{equation}
\end{proof}

\begin{lemma}\label{sharp-defect}
Set $s=|T|^2$, $k=k(s)$ and
\begin{equation}\label{alphaeta}
\hat\alpha=\frac{(n-2)k+2(n-1)sk'}n,\qquad
\hat\eta=\frac{-(n-2)k+2sk'}n.
\end{equation}
Define the nonnegative continuous function
\begin{equation}\label{Lambda}
\Lambda(s)=\max\left\{k^2,\ \hat\alpha^2,\
\frac12\left(\hat\alpha^2+\hat\eta^2+
\frac{4(n-2)}{n^2}(k+sk')^2\right)\right\}.
\end{equation}
Then, pointwise on every CMC hypersurface with the tensor $B$ of \eqref{Bdef},
\begin{equation}\label{defect-bound}
\frac12|D|^2+|\diver B|^2
\leq\frac{s(1-s)}{4h^2}\,\Lambda(s)\,|B|^2.
\end{equation}
Moreover, $\Lambda$ has a radial description independent of $h$. For $r>0$ put $\Gamma(r)=wF_c'/(2w')$, and let $\hat\Lambda(r)$ be defined by \eqref{alphaeta}--\eqref{Lambda} with $(k,sk')$ replaced by $(F_c(r),\Gamma(r))$. Then $\Lambda(s)=\hat\Lambda(r(s))$ for $s\in(0,1]$. Let $r_{\HHH}$ be defined by $nw_{-1}(r_{\HHH})=1$. Then
\begin{equation}\label{lambda-equals-k}
\hat\Lambda(r)=F_c(r)^2\quad\text{for } 0<r\leq\pi/2\ (c=1),\ \ 0<r\leq r_{\HHH}\ (c=-1).
\end{equation}
Hence $\Lambda=k^2$ on $[0,1]$ whenever $r_0(h)\leq\pi/2$ $(c=1)$, respectively $r_0(h)\leq r_{\HHH}$ $(c=-1)$. These assertions concern only the model coefficient; they do not require $B$ to be the model tensor.
\end{lemma}
\begin{proof}
At a point with $s>0$ choose $e_1=T/\sqrt s$ and complete it to an orthonormal frame $e_1,e_2,\dots,e_n$; indices $a,b$ range over $2,\dots,n$. Write
\[
p=B_{11},\qquad v_a=B_{1a},\qquad
B_{ab}=R^0_{ab}-\frac{p}{n-1}\delta_{ab},\qquad
\sum_{a}R^0_{aa}=0.
\]
Substitution in \eqref{linearD}, without taking norms first, gives
\begin{align*}
D(e_1,e_a)&=\frac{\nu\sqrt s}{2h}\left\{
\hat\alpha\, v_a e_1+k\sum_{b}B_{ab}e_b
+\frac{2(k+sk')}{n}\,p\, e_a\right\},\\
D(e_a,e_b)&=\frac{\nu\sqrt s}{h}\,\frac{k+sk'}n\,
(v_a e_b-v_b e_a).
\end{align*}
Taking the contraction $(\diver B)_j=\sum_i\langle D(e_i,e_j),e_i\rangle$, or directly from \eqref{divB}, yields
\[
(\diver B)_1=-\frac{\nu\sqrt s}{2h}\,\hat\alpha\, p,
\qquad (\diver B)_a=\frac{\nu\sqrt s}{2h}\,\hat\eta\, v_a.
\]
In the first line of the formula for $D$, the transverse block $kB_{ab}+\frac{2(k+sk')}{n}p\,\delta_{ab}$ has trace-free part $kR^0$ and scalar part $\hat\alpha p/(n-1)$. Summing over ordered pairs $(i,j)$, including both orders, and using $\nu^2=1-s$, we get
\begin{align*}
|D|^2&=\frac{s(1-s)}{2h^2}\Big\{
k^2|R^0|^2+\frac{\hat\alpha^2}{n-1}p^2+\Big[\hat\alpha^2+\frac{4(n-2)}{n^2}(k+sk')^2\Big]|v|^2
\Big\},\\
|\diver B|^2&=\frac{s(1-s)}{4h^2}
\left\{\hat\alpha^2p^2+\hat\eta^2|v|^2\right\}.
\end{align*}
Since $|B|^2=|R^0|^2+\frac{n}{n-1}p^2+2|v|^2$, comparison of these three orthogonal blocks proves \eqref{defect-bound}. Smoothness covers $s=0$, and the factor $1-s=\nu^2$ covers $s=1$.

It remains to prove the radial assertions. Since $k=F_c\circ r$, $s=n^2h^2w^2$ and $ds/dr=2n^2h^2ww'>0$, we have $sk'(s)=\Gamma(r(s))$, hence $\Lambda(s)=\hat\Lambda(r(s))$; the value at $s=0$ follows by continuity. {By Lemma \ref{bounds}(f), $F_c'<0$, that is, $\Gamma<0$, on the whole domain.} After division by $F_c^2$, the three entries in the definition of $\hat\Lambda$ are
\begin{equation}\label{three-entries}
1,\qquad \hat A(\chi)^2,\qquad
\frac{\hat A(\chi)^2+\hat E(\chi)^2+4(n-2)(1+\chi)^2/n^2}{2},
\end{equation}
where $\chi=\Gamma/F_c$, $\hat A(\chi)=[n-2+2(n-1)\chi]/n$ and $\hat E(\chi)=[-(n-2)+2\chi]/n$. The last two expressions are convex quadratic functions of $\chi$. On $[-1,0]$ and on $[0,1/n]$ their endpoint values do not exceed $1$: at $\chi=-1$ both equal $1$; at $\chi=0$ they equal $((n-2)/n)^2$ and $(n-2)/n$; at $\chi=1/n$ they equal $(1-2/n^2)^2<1$ and $1-2(n+1)/n^3<1$. Hence it is enough to show that $\chi$ lies in one of these intervals.

\emph{Spherical case.} By Lemma \ref{bounds}(b), $d_+=1-nmw=\frac n2w^2F_1>0$ on $(0,\pi)$, so $F_1=2d_+/(nw^2)>0$. By \eqref{dprime}, $d_+'>0$ on $(0,\pi/2)$. Differentiating $F_1=2d_+/(nw^2)$ gives
\[
\chi=\frac{\Gamma}{F_1}=\frac{w\,d_+'}{2w'd_+}-1 .
\]
The first term is positive, and $\chi<0$ because $\Gamma<0$ and $F_1>0$. Thus $-1<\chi<0$ on $(0,\pi/2]$, the endpoint by continuity of $d_+'$ and $\Gamma$ and by $d_+'(\pi/2)=-nw\zeta>0$, which holds since $\zeta=-nm^2-1+m/w=-1<0$ at $m=\cot(\pi/2)=0$.

\emph{Hyperbolic case.} By Lemma \ref{bounds}(d), $F_{-1}<F_0<0$, so $d_-=nmw-1=-\frac n2w^2F_{-1}>0$ and $F_{-1}=-2d_-/(nw^2)<0$; by \eqref{dprime}, $d_-'>0$. Thus $\chi=\Gamma/F_{-1}>0$, and as above $\chi=wd_-'/(2w'd_-)-1$. We show that $\chi<1/n$ on $0<r\leq r_{\HHH}$, where $w\leq1/n$. In this case \eqref{odew} gives
\begin{equation}\label{wd-ode}
w'=\frac{1-(n-1)d_-}{n},\qquad d_-'=n(w-md_-).
\end{equation}
First set $\mathcal U=w^2-d_-/n$. The expansions in the proof of Lemma \ref{regularity} give
\begin{equation}\label{Uexp}
\mathcal U(r)=\frac{2r^2}{n^2(n+2)}+O(r^4)>0
\end{equation}
near zero. At a positive zero of $\mathcal U$, \eqref{wd-ode} gives $\mathcal U'=w[2/n-(n-2)w^2]>0$, because $w\leq1/n$. A first zero approached from positive values could not have positive derivative. Consequently
\begin{equation}\label{dsmall}
0<d_-<nw^2\leq1/n .
\end{equation}
Now define $\mathcal J=2(n+1)w'd_--nwd_-'$. With
$\mathfrak a_n=(n^2+2n+2)/n$ and $\mathfrak b_n=(n^2-2)/n$, direct substitution of \eqref{wd-ode} gives
\begin{equation}\label{Ipoly}
\mathcal J=\mathfrak a_nd_--\mathfrak b_nd_-^2-n^2w^2
=\frac{2r^2}{n(n+2)}+O(r^4)>0
\end{equation}
near zero. For clarity, write $d=d_-$, $\mathfrak a=\mathfrak a_n$, and $\mathfrak b=\mathfrak b_n$. Since $d=nmw-1$, equation \eqref{wd-ode} gives
\[
w'=\frac{1-(n-1)d}{n},\qquad
d'=n(w-md)=nw-\frac{d(1+d)}w.
\]
Differentiating \eqref{Ipoly}, we obtain
\[
\mathcal J'=(\mathfrak a-2\mathfrak bd)\left(nw-\frac{d(1+d)}w\right)
-2nw[1-(n-1)d].
\]
At a positive zero of $\mathcal J$ we have $n^2w^2=\mathfrak a d-\mathfrak b d^2$. Multiplying by $w$ and substituting give
\[
\begin{aligned}
w\mathcal J'
&=\frac d n\Big\{(\mathfrak a-2\mathfrak bd)[\mathfrak a-\mathfrak bd-n(1+d)]-2(\mathfrak a-\mathfrak bd)[1-(n-1)d]\Big\}\\
&=\frac{2d}{n^2}[1-(n-1)d][\mathfrak a-(n+2)\mathfrak bd].
\end{aligned}
\]
The last equality follows by substituting the values of $\mathfrak a_n$ and $\mathfrak b_n$ and collecting powers of $d$. Dividing by $w>0$ gives
\begin{equation}\label{Iat-zero}
\mathcal J'=\frac{2d_-[1-(n-1)d_-]}{n^2w}
\bigl[\mathfrak a_n-(n+2)\mathfrak b_nd_-\bigr]>0.
\end{equation}
Indeed, $1-(n-1)d_-=nw'>0$, while \eqref{dsmall} implies
$\mathfrak a_n-(n+2)\mathfrak b_nd_->\mathfrak a_n-(n+2)\mathfrak b_n/n=4(n+1)/n^2>0$.
As for $U$, $\mathcal J$ cannot have a first positive zero. Thus $\mathcal J>0$, which is equivalent to
\[
0<\chi=\frac{w\,d_-'}{2w'd_-}-1<\frac1n .
\]
By \eqref{three-entries}, this finishes the hyperbolic case.
\end{proof}

\begin{lemma}\label{gauss-polynomial}
Put $u=|B|$, $s=|T|^2$, $b=b(s)$ and $\zeta=\zeta(s)$. At a point where $B$ has eigenvalues $\beta_i$,
\begin{equation}\label{exact-gauss}
\begin{aligned}
\sum_{i<j}K_{ij}(\beta_i-\beta_j)^2={}&
n(c+b^2)u^2+\zeta\bigl(su^2+n|BT|^2\bigr)\\
&+\frac n2\tr(B^3)-\frac{k}{2h^2}u^2\ip{BT}{T}-\frac{u^4}{4h^2}.
\end{aligned}
\end{equation}
Let
\begin{equation}\label{qPdef}
\begin{aligned}
q_{n,1,h}(s)&=n\bigl[1-s+a(s)b(s)\bigr]
-\frac{s(1-s)}{4h^2}\Lambda(s),\\
q_{n,-1,h}(s)&=\bigl[-(1-s)+a(s)b(s)\bigr]
+(n-1)\bigl[-1+b(s)^2\bigr]-\frac{s(1-s)}{4h^2}\Lambda(s),\\
\mathcal P_{n,c,h}(s)&=\frac{n(n-2)}{2\sqrt{n(n-1)}}
+\frac{|k(s)|\,s}{2h^2}\sqrt{\frac{n-1}{n}}.
\end{aligned}
\end{equation}
On the ranges \eqref{gap-range}, the integrand of \eqref{main-int} is bounded below by
\begin{equation}\label{polynomial-lower}
|\nabla B|^2+u^2\left[q_{n,c,h}(s)-\mathcal P_{n,c,h}(s)\,u
-\frac{u^2}{4h^2}\right].
\end{equation}
\end{lemma}
\begin{proof}
Since $b=h-\frac{ks}{2hn}$ and $\zeta=\frac{kb}{2h}-c$, a direct computation gives
\[
n(c+b^2)+s\zeta=nh^2+c(n-s)-\frac{ks}{2},\qquad
n\zeta=n\Big(\frac k2-c\Big)-\frac{k^2s}{4h^2}.
\]
Hence \eqref{exact-gauss} is \eqref{curvexact} with its terms regrouped.

We use the bounds recalled in the proof of Lemma \ref{curvterm}:
\begin{equation}\label{okumura}
\begin{gathered}
\tr(B^3)\geq-\frac{n-2}{\sqrt{n(n-1)}}u^3,\\
|BT|^2\leq\frac{n-1}{n}su^2,\qquad
|\ip{BT}{T}|\leq\sqrt{\frac{n-1}{n}}\,su.
\end{gathered}
\end{equation}
By Lemma \ref{model-signs}, $\zeta<0$ in the spherical range \eqref{gap-range}, because $h>h_{\SSS}(n)$ implies $r_0<\pi/2$ (see the proof of Lemma \ref{positive-q}), and $\zeta>0$ in the hyperbolic range. If $\zeta<0$, \eqref{okumura} bounds $\zeta(su^2+n|BT|^2)$ below by $n\zeta su^2$; if $\zeta>0$, it is bounded below by $\zeta su^2$. The model sectional curvatures of the meridian planes and of the orbit planes are
\[
K_{\rm mer}=c(1-s)+ab,\qquad K_{\rm orb}=c+b^2,\qquad
K_{\rm mer}-K_{\rm orb}=s\zeta,
\]
where the last identity follows from \eqref{ab} and \eqref{zetadef}. Thus the quadratic terms of \eqref{exact-gauss} are bounded below by $nK_{\rm mer}u^2$ if $c=1$ and by $[K_{\rm mer}+(n-1)K_{\rm orb}]u^2$ if $c=-1$. The remaining terms are estimated by \eqref{okumura}, and the defect terms of \eqref{main-int} by \eqref{defect-bound}. This gives \eqref{polynomial-lower} with precisely \eqref{qPdef}.
\end{proof}

\begin{lemma}\label{positive-q}
Under \eqref{gap-range}, $q_{n,c,h}(s)>0$ for $0\leq s\leq1$. In the hyperbolic case, in fact,
\begin{equation}\label{hyper-qmin}
\min_{[0,1]}q_{n,-1,h}=n(h^2-1).
\end{equation}
At $h=1$ the hyperbolic coefficient vanishes at the poles; at $h=h_{\SSS}(n)$ the spherical coefficient vanishes at the equator. Below either threshold it is negative at the corresponding point.
\end{lemma}
\begin{proof}
First suppose $c=1$. Since $w(\pi/2)=\int_0^{\pi/2}\sin^{n-1}r\,dr$ and $w$ is strictly increasing, $h>h_{\SSS}(n)$ is equivalent to $nhw(\pi/2)>1$, that is, to $r_0(h)<\pi/2$. Hence $a=nhw'>0$ and $b=nhwm>0$ throughout the profile. Let $d_+=1-nmw>0$ as in the proof of Lemma \ref{sharp-defect}. The function $v=(nm+\sqrt n)^{-1}$ satisfies
\[
1-[v'+(n-1)mv]
=\frac{(n+1)\sqrt n\,m}{(nm+\sqrt n)^2}>0
\qquad(0<r<\pi/2).
\]
As in the proof of Lemma \ref{model-signs}, $(f_1^{n-1}(w-v))'>0$ and $f_1^{n-1}(w-v)\to0$ as $r\downarrow0$, so $w>v$, which is equivalent to $d_+<\sqrt n\,w$. By \eqref{lambda-equals-k}, $k=2d_+/(nw^2)$ and $s=(nhw)^2$, we obtain the explicit expression
\begin{equation}\label{sphere-qpositive}
q_{n,1,h}(s)=nab+(1-s)\left(n-\frac{d_+^2}{w^2}\right)>0
\qquad(0<r\leq r_0).
\end{equation}
At the pole its continuous value is $n(1+h^2)>0$. At $h=h_{\SSS}(n)$ we have $r_0=\pi/2$ and $q_{n,1,h}(1)=na(r_0)b(r_0)=0$, since $b(r_0)=m(r_0)=\cot r_0$. For $h<h_{\SSS}(n)$ we have $r_0>\pi/2$, $b(r_0)=\cot r_0<0$ and $q_{n,1,h}(1)=na(r_0)b(r_0)<0$.

Now let $c=-1$. Define $r_{\HHH}$ by $nw(r_{\HHH})=1$, and on $0<r\leq r_{\HHH}$ put $d_-=nmw-1>0$ and $Y=nw$, so that $0<Y\leq1$. By \eqref{dsmall}, $nd_-<Y^2$. At $h=1$ we have $s=Y^2$, $b=nwm=1+d_-$ and $a=nw'=1-(n-1)d_-$, so the model curvatures satisfy
\[
K_{\rm mer}+(n-1)K_{\rm orb}=Y^2+nd_- .
\]
Because $k=-2nd_-/Y^2$ and $\Lambda=k^2$ by \eqref{lambda-equals-k}, the defect subtraction is
\[
\frac{Y^2(1-Y^2)}4\,k^2
=\frac{n^2d_-^2}{Y^2}(1-Y^2)<nd_-(1-Y^2).
\]
It follows that
\begin{equation}\label{hyper-critical}
q_{n,-1,1}(r)>Y^2(1+nd_-)>0\quad(0<r\leq r_{\HHH}),
\qquad q_{n,-1,1}(0)=0.
\end{equation}
For $h>1$ we have $r_0(h)<r_{\HHH}$. At each fixed $r\leq r_0(h)$, direct substitution of $s=h^2Y^2$, $a=h[1-(n-1)d_-]$, $b=h(1+d_-)$ and $k=-2nd_-/Y^2$ into \eqref{qPdef} gives
\begin{equation}\label{hyper-h-growth}
q_{n,-1,h}(r)-q_{n,-1,1}(r)
=(h^2-1)\bigl[Y^2+n(1+d_-)+n^2d_-^2\bigr]
\geq n(h^2-1).
\end{equation}
The inequality is strict for $r>0$, and the value at the pole equals $n(h^2-1)$. This proves \eqref{hyper-qmin}. For every admissible $h<1$, the pole value $n(h^2-1)$ is negative.
\end{proof}

\begin{remark}\label{sharp-convexity}
For $c=1$ the principal curvatures of the lower hemisphere of the reference sphere are $a=nhw'>0$ (Lemma \ref{bounds}(c)) and $b=nhwm$, whose sign is that of $\cot r$ on $(0,r_0]$; the upper hemisphere is its reflection. Hence, by the first paragraph of the proof of Lemma \ref{positive-q}, the reference sphere in $\SSS^n\times\RR$ is strictly convex if and only if $h>h_{\SSS}(n)$, and for $h=h_{\SSS}(n)$ its principal curvature $b$ vanishes along the equator. For $c=-1$ the reference sphere is strictly convex in the whole range \eqref{range}, since $m=\coth r>0$, while the method requires $h>1$.
\end{remark}

\begin{proof}[Proof of Theorem~\ref{main-sharp}]
Fix $n,c,h$ in the range \eqref{gap-range}, and write $q(s)=q_{n,c,h}(s)$ and $\mathcal P(s)=\mathcal P_{n,c,h}(s)$. Lemma~\ref{positive-q} gives $\min_{[0,1]}q>0$. For each $s\in[0,1]$, the function
\[
\Phi_s(u)=q(s)-\mathcal P(s)u-\frac{u^2}{4h^2},\qquad u\geq0,
\]
is strictly decreasing from $\Phi_s(0)>0$ and has a unique positive zero. Solving $\Phi_s(u)=0$ gives the expression inside the minimum in \eqref{radius}. Its continuous dependence on $s$ shows that $R_*>0$.

Set $M_B=\sup_\Sigma|B|<R_*$. Since $\Phi_s$ is decreasing,
\[
\Phi_s(u)\geq\Phi_s(M_B)>0
\qquad(0\leq u\leq M_B,\ 0\leq s\leq1).
\]
Continuity and compactness of $[0,1]$ give $\varepsilon_*:=\min_{s\in[0,1]}\Phi_s(M_B)>0$. Applying \eqref{polynomial-lower} to the integrated identity \eqref{main-int}, we obtain
\[
0\geq\int_\Sigma\bigl(|\nabla B|^2+\varepsilon_*|B|^2\bigr).
\]
Hence $B\equiv0$, and Theorem~\ref{main-equality} identifies $\Sigma$ with the rotational CMC sphere.
\end{proof}

We now compare the pinching \eqref{pinch-sharp} with convexity. Put $\eta_0=\sqrt{(n-1)/n}$, $\eta_1=1/\sqrt{n(n-1)}$,
\[
\mathfrak m(s)=\min\{a(s),b(s)\},\qquad
X_{\rm conv}(n,c,h)=\frac{2h}{\eta_0}\min_{[0,1]}\mathfrak m .
\]

\begin{proposition}\label{convex-criterion}
Let $n\geq3$ and let $h$ be in \eqref{gap-range}.
\begin{enumerate}
\item[(i)] $\mathfrak m>0$ on $[0,1]$. At a point with $|T|^2=s$ and $|B|<2h\mathfrak m(s)/\eta_0$, every principal curvature is positive. In particular, if $R_*(n,c,h)\leq X_{\rm conv}(n,c,h)$, every hypersurface satisfying the hypotheses of Theorem \ref{main-sharp} is strictly convex.
\item[(ii)] Conversely, for $s\in[0,1]$ and $x>2h\mathfrak m(s)/\eta_0$ there are a trace-free symmetric $B$ with $|B|=x$ and a vector $T$ with $|T|^2=s$ such that the operator
\[
A=h\Id+\frac1{2h}B+\frac{k(s)}{2h}\Big(T\otimes T-\frac sn\Id\Big)
\]
determined by \eqref{Bdef} has a negative eigenvalue. In particular, if $R_*(n,c,h)>X_{\rm conv}(n,c,h)$, the pinching region of Theorem \ref{main-sharp} contains such algebraic data.
\end{enumerate}
\end{proposition}
\begin{proof}
By \eqref{Bdef}, $A=A_0+\frac1{2h}B$, where $A_0=h\Id+\frac{k}{2h}(T\otimes T-\frac sn\Id)$. By \eqref{ab}, $A_0T=a(s)T$ and $A_0X=b(s)X$ for $X\perp T$; when $s=0$, $A_0=h\Id$ and $a(0)=b(0)=h$. We have $a=nhw'>0$ by Lemma \ref{bounds}(c), and $b=nhwm>0$ because $m>0$ on $(0,r_0]$: for $c=-1$ always, and for $c=1$ because $r_0<\pi/2$ (proof of Lemma \ref{positive-q}). So $\mathfrak m>0$.

(i) Since $B$ is trace-free, $\lambda_{\min}(B)\geq-\eta_0|B|$ (proof of Lemma \ref{curvterm}). Hence $\lambda_{\min}(A)\geq\mathfrak m(s)-\eta_0|B|/(2h)>0$. If $|B|\leq\sup_\Sigma|B|<R_*\leq X_{\rm conv}\leq2h\mathfrak m(s)/\eta_0$ everywhere, $\Sigma$ is strictly convex.

(ii) Let $e$ be a unit vector with $A_0e=\mathfrak m(s)e$: $e=T/\sqrt s$ if $s>0$ and $\mathfrak m(s)=a(s)$, and $e\perp T$ otherwise (for $s=0$, $T=0$ and $e$ is arbitrary). Put $B=x\big(-\eta_0\,e\otimes e+\eta_1(\Id-e\otimes e)\big)$. Then $\tr B=x(-\eta_0+(n-1)\eta_1)=0$ and $|B|^2=x^2(\eta_0^2+(n-1)\eta_1^2)=x^2$. Moreover
\[
\ip{Ae}{e}=\mathfrak m(s)-\frac{\eta_0x}{2h}<0 .
\]
If $R_*>X_{\rm conv}$, choose $s$ minimizing $\mathfrak m$ and $x\in(X_{\rm conv},R_*)$.
\end{proof}

\begin{remark}\label{sharp-remark}
Proposition \ref{convex-criterion} gives the precise pointwise convexity test for the threshold $R_*$. The comparison is carried out below: $R_*<X_{\rm conv}$ for every $h>1$ in $\HHH^n\times\RR$, and $R_*>X_{\rm conv}$ in $\SSS^n\times\RR$ on the range of Theorem \ref{main-compare}(ii); see Remark \ref{algebraic-remark} for the meaning of the latter inequality.
\end{remark}

\subsection*{Comparison with convexity for all admissible $h$}
For $u\geq0$ and $s\in[0,1]$ put, as in the proof of Theorem \ref{main-sharp},
\[
\Phi_s(u)=q_{n,c,h}(s)-\mathcal P_{n,c,h}(s)\,u-\frac{u^2}{4h^2},\qquad p_0=\frac{n(n-2)}{2\sqrt{n(n-1)}} .
\]
Each $\Phi_s$ is strictly decreasing, and its unique positive zero $R_h(s)$ is the expression inside the minimum in \eqref{radius}. Hence, for $X>0$, $R_h(s)>X$ if and only if $\Phi_s(X)>0$. We write $\bar a=a(1)$, $\bar b=b(1)$ and $\bar k=k(1)$.

\begin{lemma}\label{equator}
Let $n\geq3$ and let $h$ satisfy \eqref{gap-range}.
\begin{enumerate}
\item[(i)] If $c=1$, then $\mathfrak m=b$, $b$ is decreasing in $s$, and $X_{\rm conv}=2hb_1/\eta_0$. Moreover
\begin{equation}\label{equator-identity}
\Phi_1(X_{\rm conv})=\frac{(n-2)^2}{2(n-1)h}\,\bar k\bar b>0,
\end{equation}
and, with $\beta(s)=b(s)-\bar b\geq0$,
\begin{equation}\label{equator-difference}
\Phi_s(X_{\rm conv})-\Phi_1(X_{\rm conv})=n\,\beta(s)\big(a(s)-(n-3)\bar b\big)+(1-s)\Big(n-\frac{k(s)^2s}{4h^2}\Big).
\end{equation}
\item[(ii)] If $c=-1$, then $\mathfrak m=a$, $a$ is decreasing in $s$, and $X_{\rm conv}=2ha_1/\eta_0$. Moreover, with $\bar\delta=h-\bar a=\frac{(n-1)|\bar k|}{2hn}$,
\begin{equation}\label{pole-identity}
\Phi_0(X_{\rm conv})=n\Big(\frac{|\bar k|}{2}-1\Big)-\frac{n\,\bar\delta^2}{n-1}.
\end{equation}
\end{enumerate}
\end{lemma}

\begin{proof}
By \eqref{ab}, $a-b=ks/(2h)$, $b=h-ks/(2hn)$ and $nh=a+(n-1)b$. Recall also $p_0/\eta_0=n(n-2)/(2(n-1))$ and $1/\eta_0^2=n/(n-1)$.

(i) Let $c=1$. Then $h>h_{\SSS}(n)$ gives $r_0<\pi/2$ (proof of Lemma \ref{positive-q}), so $\Lambda=k^2$ by \eqref{lambda-equals-k}, and $k>0$ by Lemma \ref{bounds}(d). Thus
\[
q_{n,1,h}(s)=nab+(1-s)\Big(n-\frac{k^2s}{4h^2}\Big),\qquad \mathcal P_{n,1,h}(s)=p_0+\frac{\eta_0ks}{2h^2}.
\]
By the proof of Lemma \ref{sharp-defect} (spherical case), $\chi=sk'/k\in(-1,0)$, so $0<(ks)'=k(1+\chi)<k$. Hence $ks$ is increasing, $b$ is decreasing, and $a\geq b$; therefore $\mathfrak m=b$ and $X:=X_{\rm conv}=2hb_1/\eta_0$. Using $ks=2h(a-b)$,
\[
\mathcal P(s)X=\frac{n(n-2)}{n-1}\,hb_1+2(a-b)\bar b,\qquad \frac{X^2}{4h^2}=\frac{n}{n-1}\,\bar b^2 .
\]
At $s=1$ the defect term disappears, and substituting $h=(\bar a+(n-1)\bar b)/n$,
\[
\Phi_1(X)=\bar b\Big[na_1-\frac{n-2}{n-1}\big(\bar a+(n-1)\bar b\big)-2(\bar a-\bar b)-\frac{n}{n-1}\bar b\Big]=\frac{(n-2)^2}{n-1}\,\bar b(\bar a-\bar b),
\]
which is \eqref{equator-identity} since $\bar a-\bar b=\bar k/(2h)$. For \eqref{equator-difference}, note $a-\bar a=-(n-1)\beta$ and $(a-b)-(\bar a-\bar b)=-n\beta$. Then $ab-\bar ab_1=a\beta-(n-1)\bar b\beta$, and
\[
\Phi_s(X)-\Phi_1(X)=n(ab-\bar ab_1)-2\bar b\big[(a-b)-(\bar a-\bar b)\big]+(1-s)\Big(n-\frac{k^2s}{4h^2}\Big),
\]
which equals the right side of \eqref{equator-difference}. Finally $\bar b>0$ (proof of Proposition \ref{convex-criterion}) and $\bar k>0$.

(ii) Let $c=-1$. Then $k<0$ and $k'<0$ (Lemma \ref{bounds}(d),(f)), so $(ks)'=k+sk'<0$; hence $a=h-(n-1)|k|s/(2hn)$ is decreasing and $a\leq b$. Therefore $\mathfrak m=a$ and $X=2ha_1/\eta_0$. At the pole $q_{n,-1,h}(0)=n(h^2-1)$ by \eqref{hyper-qmin}, and $\mathcal P(0)=p_0$. Thus
\[
\Phi_0(X)=nh^2-n-\frac{n(n-2)}{n-1}\,ha_1-\frac{n}{n-1}\,\bar a^2 .
\]
Substituting $\bar a=h-\bar\delta$, the terms in $h^2$ cancel and $\Phi_0(X)=-n+\frac{n^2}{n-1}h\bar\delta-\frac{n}{n-1}\bar\delta^2$. Since $\frac{n^2}{n-1}h\bar\delta=\frac n2|\bar k|$, this is \eqref{pole-identity}.
\end{proof}

\begin{proof}[Proof of Theorem \ref{main-compare}]
(i) By Lemma \ref{bounds}(d), $|\bar k|<\mathcal K_{n,-1}<2$, so \eqref{pole-identity} gives $\Phi_0(X_{\rm conv})<0$. Hence $R_*\leq R_h(0)<X_{\rm conv}$, and Proposition \ref{convex-criterion}(i) gives strict convexity.

(ii) Let $c=1$ and $X=X_{\rm conv}$. By \eqref{equator-identity} it suffices to show that the right side of \eqref{equator-difference} is nonnegative for $s\in[0,1)$; then $\Phi_s(X)>0$, that is, $R_h(s)>X$, for every $s\in[0,1]$, and $R_*=\min_sR_h(s)>X$ by compactness. With $d_+=1-nmw$ as in Lemma \ref{sharp-defect}, $k=2d_+/(nw^2)$ and $s=n^2h^2w^2$ give $k^2s/(4h^2)=d_+^2/w^2<n$ (proof of Lemma \ref{positive-q}); so the second term of \eqref{equator-difference} is positive for $s<1$. Moreover $\beta\geq0$ and $a\geq h\geq \bar b$, since $k\geq0$.

If $n=3$, then $a-(n-3)\bar b=a>0$; if $n=4$, then $a-\bar b\geq0$. In both cases the first term of \eqref{equator-difference} is nonnegative.

Let $n\geq5$. Then $a-(n-3)\bar b\geq-(n-4)h$. Since $0<(ks)'<k<F_0$ (Lemma \ref{bounds}(d)),
\[
\beta(s)=\frac{\bar k-k(s)s}{2hn}=\frac1{2hn}\int_s^1(k+tk')\,dt\leq\frac{F_0(1-s)}{2hn},
\]
so the first term is at least $-(n-4)F_0(1-s)/2=-(1-s)\,n(n-4)/(n+2)$. Using $k<F_0$ in the second term,
\[
\Phi_s(X)-\Phi_1(X)\geq(1-s)\,\frac{n}{n+2}\Big(6-\frac{n}{(n+2)h^2}\Big)\geq0
\]
whenever $h^2\geq n/(6(n+2))$, which is the hypothesis of (ii). The last assertion of (ii) follows from Proposition \ref{convex-criterion}(ii).
\end{proof}

\begin{remark}\label{window}
Since $\Phi_s$ is strictly decreasing, \eqref{equator-identity} also gives the quantitative bound $R_h(1)-X_{\rm conv}=\Phi_1(X_{\rm conv})/\big(\mathcal P(1)+(R_h(1)+X_{\rm conv})/(4h^2)\big)$; as $h\to\infty$ this tends to $\frac{2(n-2)^2}{(n+2)\sqrt{n(n-1)}}$, in agreement with Theorem \ref{main-asymp}(c). For $n\geq5$ the window $h_{\SSS}(n)<h<\sqrt{n/(6(n+2))}$ is not covered by Theorem \ref{main-compare}(ii); it is narrow, for instance $0.3395<h<0.3450$ when $n=5$. There, numerical integration of the profile gives $R_*-X_{\rm conv}>0$; for instance, at the midpoint of the window one finds $0.0030$, $0.0271$, $0.0791$ and $0.1302$ for $n=5,6,8,10$, and at every sampled $h$ the minimum of $s\mapsto\Phi_s(X_{\rm conv})$ is attained at $s=1$. We state this only as evidence: a proof in this window would require a sharper bound than $k\leq F_0$ in \eqref{equator-difference}.
\end{remark}

\section{Joint optimization of the algebraic estimate}\label{opt-section}
In the proof of \eqref{curvlower}, several terms of \eqref{curvexact} were estimated separately: Okumura's inequality for $\sum\beta_i^3$, and the bounds for $|BT|^2$ and $\ip{BT}{T}$. These estimates are not attained simultaneously. The cubic term is extremal when $B$ has a simple eigenvalue of sign opposite to its $(n-1)$-fold eigenvalue. The worst sign of $\ip{BT}{T}$, however, requires $T$ along a direction where the other terms are favorable. We now minimize the algebraic part of the estimate jointly. Two relaxations remain: the coefficient $k=k(s)$ is replaced by an interval of values independent of $s$, and the defect terms are still controlled by the uniform bound \eqref{error1}. The equality case of Okumura's inequality \cite[Lemma (2.6)]{AdC}, applied to the eigenvalues other than a prescribed one, makes the reduction exact.

\begin{lemma}\label{okumura-y}
Let $n\geq3$, {$\bar n=n-1$, $c_{\bar n}=(\bar n-2)/\sqrt{\bar n(\bar n-1)}$}, $x>0$ and $y^2\leq\frac{n-1}{n}x^2$. Among $\beta\in\RR^n$ with $\sum_i\beta_i=0$, $\sum_i\beta_i^2=x^2$ and $\beta_n=y$,
\begin{equation}\label{Cn}
\min\sum_i\beta_i^3=C_n(x,y):=\Big(1-\frac1{\bar n^2}\Big)y^3-\frac{3y}{\bar n}\varrho^2-c_{\bar n}\varrho^3,
\end{equation}
where {$\varrho^2=x^2-\frac{n}{n-1}y^2$}. The minimum is attained exactly when $\beta_i=-\frac y{\bar n}+z_i$ ($i<n$), where {$z\in\RR^{\bar n}$} has $\sum z_i=0$, $|z|=\varrho$, and $\bar n-1$ equal nonnegative entries (for $\bar n=2$, any such $z$).
\end{lemma}
\begin{proof}
Write $\beta_i=\bar y+z_i$ for $i<n$, with {$\bar y=-y/\bar n$} and $\sum_{i<n}z_i=0$. Then $\sum_{i<n}z_i^2=x^2-y^2-\bar n\bar y^2=\varrho^2$, and
\[
\sum_i\beta_i^3=y^3+\bar n\bar y^3+3\bar y\varrho^2+\sum_{i<n}z_i^3 .
\]
Okumura's inequality \cite{Okumura}, in the form proved by Alencar and do Carmo \cite[Lemma (2.6)]{AdC}, gives $\sum z_i^3\geq-c_{\bar n}\varrho^3$ in $\RR^{\bar n}$. By the equality statement of \cite[Lemma (2.6)]{AdC}, equality in the lower bound holds if and only if $\bar n-1$ of the $z_i$ are nonnegative and equal. Substituting $\bar y$ gives \eqref{Cn}.
\end{proof}

Let $F_0$, $\mathcal K$, $Q$ be as in Lemma \ref{bounds}. By Lemma \ref{bounds}(d) and $w\leq1/(nh)$ on $(0,r_0]$, the coefficient takes values in
\begin{equation}\label{kint}
I_{n,c}(h)=[k_{\rm lo},F_0],\qquad
k_{\rm lo}=\begin{cases}\max\{0,\,F_0-Q/(n^2h^2)\},&c=1,\\ \max\{-\mathcal K,\,F_0-Q/(n^2h^2)\},&c=-1.\end{cases}
\end{equation}
For $x>0$, $s\in[0,1]$, $k\in\RR$ and $|y|\leq\sqrt{(n-1)/n}\,x$, put
\begin{equation}\label{Gdef}
\begin{aligned}
G(y,s,k;x)={}&\frac n2C_n(x,y)+\Big(c(n-s)-\frac{ks}{2}\Big)x^2\\
&+s\Big(n\big(\tfrac k2-c\big)-\frac{k^2s}{4h^2}\Big)y^2-\frac{ksx^2}{2h^2}\,y ,
\end{aligned}
\end{equation}
and
\begin{equation}\label{Thetadef}
\Theta(x)=\min\Big\{G(y,s,k;x):\ s\in[0,1],\ k\in I_{n,c}(h),\ |y|\leq\sqrt{\tfrac{n-1}{n}}\,x\Big\}.
\end{equation}
The function $G$ is concave in $k$ (the coefficient of $k^2$ is $-s^2y^2/(4h^2)\leq0$) and concave in $s$, since $\partial_s^2G=-k^2y^2/(2h^2)\leq0$. For fixed $y$, first minimize in $s$ for fixed $k$, then in $k$ for fixed $s\in\{0,1\}$. Each minimum of a concave function of one variable over an interval is attained at an endpoint. At $s=0$, $G$ does not depend on $k$. Hence
\begin{equation}\label{Theta-vertices}
\Theta(x)=\min_{|y|\leq\eta_0x}\min\big\{G(y,0,F_0;x),\ G(y,1,k_{\rm lo};x),\ G(y,1,F_0;x)\big\},
\end{equation}
where $\eta_0=\sqrt{(n-1)/n}$. We also set $\Theta(0)=0$, which is the value of the same formula at $x=0$.

\emph{Continuity.} Write $y=x\eta$ with $|\eta|\leq\eta_0$. Since $C_n(x,x\eta)=x^3C_n(1,\eta)$, \eqref{Gdef} gives $G(x\eta,s,k;x)=x^2\,\tilde G(\eta,s,k;x)$, where
\[
\tilde G(\eta,s,k;x)=\tfrac n2\,x\,C_n(1,\eta)+c(n-s)-\frac{ks}2+s\Big(n\big(\tfrac k2-c\big)-\frac{k^2s}{4h^2}\Big)\eta^2-\frac{ksx}{2h^2}\,\eta
\]
is continuous on $[-\eta_0,\eta_0]\times[0,1]\times I_{n,c}(h)\times[0,\infty)$; here $C_n(1,\eta)$ is continuous in $\eta$ because $\varrho=(1-\frac{n}{n-1}\eta^2)^{1/2}$ is. The domain of $(\eta,s,k)$ is a fixed compact set, so $\tilde\Theta(x)=\min_{\eta,s,k}\tilde G(\eta,s,k;x)$ is continuous on $[0,\infty)$, and $\Theta(x)=x^2\tilde\Theta(x)$ for all $x\geq0$. Finally let
\[
\begin{gathered}
\Pi(x)=x^2\Big(nh^2-\frac{\varepsilon}{h^2}-\frac{x^2}{4h^2}\Big)+\Theta(x),\\
x_{\rm opt}=\sup\{X>0:\ \Pi>0\text{ on }(0,X]\},
\end{gathered}
\]
with $x_{\rm opt}=0$ if this set is empty. Note $\Pi(0)=0$, so the integrand below vanishes where $B=0$. Moreover $\Pi(x)=x^2\tilde\Pi(x)$ with $\tilde\Pi(x)=nh^2-\varepsilon/h^2-x^2/(4h^2)+\tilde\Theta(x)$ continuous on $[0,\infty)$. Hence $x_{\rm opt}>0$ whenever $\tilde\Pi(0)>0$, and if $0<x_{\rm opt}<\infty$, then $\Pi(x_{\rm opt})=0$ and $x_{\rm opt}$ is the first positive zero of $\Pi$. This is the form in which $x_{\rm opt}$ is used below.

\begin{lemma}\label{joint}
At each point of a CMC hypersurface in the range \eqref{range},
\begin{equation}\label{jointbound}
\sum_{i<j}K_{ij}(\beta_i-\beta_j)^2\geq|B|^2\Big(nh^2-\frac{|B|^2}{4h^2}\Big)+\Theta(|B|).
\end{equation}
For fixed $x=|B|$, $s$ and $k$, the reduction is exact in the algebraic problem: the right side of \eqref{curvexact}, minimized over all trace-free $B$ with $|B|=x$ and all $T$ with $|T|^2=s$, equals $x^2(nh^2-x^2/(4h^2))+\min_yG(y,s,k;x)$. This exactness does not extend to hypersurfaces, since $k$ is relaxed to $I_{n,c}(h)$ and the defect terms are estimated separately.
\end{lemma}
\begin{proof}
In the eigenframe of $B$ put $t_i=T_i^2$, so that $t_i\geq0$ and $\sum t_i=s$. For fixed $\beta$, $s$ and $k=k(s)$, the right side of \eqref{curvexact} is affine in $t$. Indeed, $|BT|^2=\sum\beta_i^2t_i$ and $\ip{BT}T=\sum\beta_it_i$. Its minimum over the simplex is therefore attained at a vertex $t=s\,\delta_{j}$, that is, with $T$ an eigenvector of $B$ with eigenvalue $y=\beta_j$. At such a vertex the right side of \eqref{curvexact} equals
\begin{multline*}
x^2\Big(nh^2-\frac{x^2}{4h^2}\Big)+\frac n2\sum_i\beta_i^3+\Big(c(n-s)-\frac{ks}2\Big)x^2
+s\Big(n\big(\tfrac k2-c\big)-\frac{k^2s}{4h^2}\Big)y^2-\frac{ksx^2}{2h^2}y .
\end{multline*}
By Lemma \ref{okumura-y}, the minimum of $\sum\beta_i^3$ with $\beta_j=y$ prescribed is $C_n(x,y)$, and it is attained. This gives \eqref{jointbound}, because $k(s)\in I_{n,c}(h)$. Exactness of the algebraic reduction follows since both minimizations are attained.
\end{proof}

\begin{theorem}\label{opt-gap}
Let $n\geq3$, let $h$ be in \eqref{range}, and assume $x_{\rm opt}>0$. If a connected, closed, two-sided CMC immersion $\Sigma^n\to\QQ_c^n\times\RR$ with normalized mean curvature $h$ satisfies $\sup_\Sigma|\mathscr S_{n,c,h}|<x_{\rm opt}$, then it is congruent to the rotational CMC sphere of mean curvature $h$. Moreover $x_{\rm opt}\geq x_+$ whenever $a_0>0$.
\end{theorem}
This is the part of Theorem \ref{main-gap} concerning $x_{\rm opt}$. The positivity assumption is automatic when $a_0>0$; if $x_{\rm opt}=0$, the pinching condition is empty. The inequality $|B|<x_{\rm opt}$ is strict. The slack argument used for $x_+$ in Theorem \ref{main-gap} is not transferred to $x_{\rm opt}$.
\begin{proof}
By \eqref{main-int}, \eqref{error1} and \eqref{jointbound}, $0\geq\int_\Sigma\{|\nabla B|^2+\Pi(|B|)\}$, and $\Pi(|B|)>0$ wherever $B\neq0$. Hence $B\equiv0$, and Theorem \ref{main-equality} concludes.

For the comparison, recall that every inequality used in the proof of \eqref{curvlower} holds for all admissible $B$, $T$, and all $k$ with $|k|\leq \mathcal K$ (and $k\geq0$ if $c=1$). This includes all $k\in I_{n,c}(h)$. Hence $\Pi(x)\geq x^2P(x)$, and $P>0$ on $(0,x_+)$.
\end{proof}

\subsection*{Asymptotics of \texorpdfstring{$x_{\rm opt}$}{x opt}}
Put
\[
\begin{gathered}
\eta_0=\sqrt{\tfrac{n-1}{n}},\quad \eta_1=\frac1{\sqrt{n(n-1)}},\quad o_n=\frac{n-2}{\sqrt{n(n-1)}},\\
\lambda=2\sqrt{\tfrac{n}{n-1}},\quad \varpi=\frac{n\sqrt n}{2\sqrt{n-1}},
\end{gathered}
\]
and $\hat C(\eta)=C_n(1,\eta)$ for $|\eta|\leq\eta_0$. Then $C_n(x,x\eta)=x^3\hat C(\eta)$.

\begin{lemma}\label{Chat}
$\min\hat C=-o_n$, and the minimum is attained exactly at $\eta=-\eta_0$ and $\eta=\eta_1$.
\end{lemma}
\begin{proof}
By Lemma \ref{okumura-y}, $\min_\eta\hat C(\eta)$ is the minimum of $\sum\beta_i^3$ over trace-free $\beta$ with $|\beta|=1$. By Okumura's inequality and the equality case in \cite[Lemma (2.6)]{AdC}, this is $-o_n$. Equality holds exactly when $n-1$ of the $\beta_i$ equal $\eta_1$ and the remaining one equals $-\eta_0$. The prescribed eigenvalue $\eta$ is then either $-\eta_0$ or $\eta_1$. Conversely, both values give $\varrho$ and $z$ realizing the extremal configuration, by the equality statement of Lemma \ref{okumura-y}.
\end{proof}

\begin{proof}[Proof of Theorem \ref{main-asymp}(a)]
\emph{Rescaling.} Write $x=\mathfrak x h^2$ and $y=x\eta$. Directly from \eqref{Gdef},
\[
G(y,s,k;x)=x^2\Big[\tfrac n2\,\mathfrak x h^2\,\hat C(\eta)+g(\eta,s,k,\mathfrak x)-\frac{k^2s^2\eta^2}{4h^2}\Big],
\]
where
\[
g=c(n-s)-\frac{ks}2+sn\Big(\frac k2-c\Big)\eta^2-\frac{ks\mathfrak x\eta}{2}.
\]
Hence $\Pi(\mathfrak x h^2)=\mathfrak x^2h^6\,\mathcal F_h(\mathfrak x)$, with
\[
\mathcal F_h(\mathfrak x)=L(\mathfrak x)+h^{-2}\mathcal G_h(\mathfrak x),\qquad L(\mathfrak x)=n-\frac{\mathfrak x^2}4-\frac{no_n}{2}\mathfrak x,
\]
\[
\mathcal G_h(\mathfrak x)=\min_{\eta,s,\,k\in I_{n,c}(h)}\Big[\tfrac n2\mathfrak x h^2\big(\hat C(\eta)+o_n\big)+g(\eta,s,k,\mathfrak x)-\frac{k^2s^2\eta^2}{4h^2}\Big]-\frac{\varepsilon}{h^2}.
\]
A direct computation gives $L(\lambda)=0$ and $L'(\lambda)=-\varpi<0$. Moreover $L$ is decreasing on $(0,\infty)$.

\emph{Step 1: $\mathcal G_h\to M$ uniformly on compact subsets $J\subset(0,\infty)$,} where
\[
M(\mathfrak x)=\min\{g(\eta,s,F_0,\mathfrak x):\ \eta\in\{-\eta_0,\eta_1\},\ s\in[0,1]\}.
\]
As $h\to\infty$, $\sigma_{n,c}(h)\to0$, so $\varepsilon$ stays bounded, and $I_{n,c}(h)\subset[F_0-Q/(n^2h^2),F_0]$.

For the upper bound, take $\eta\in\{-\eta_0,\eta_1\}$, $k=F_0$ and $s$ minimizing. This gives $\mathcal G_h\leq M+O(h^{-2})$.

For the lower bound, let $\ell$ be a Lipschitz constant of $g$ in $(\eta,k)$ on the relevant compact set, and let $\mathfrak x_J=\min J>0$. Fix $\delta>0$ and put $\omega(\delta)=\min\{\hat C(\eta)+o_n:\ \operatorname{dist}(\eta,\{-\eta_0,\eta_1\})\geq\delta\}$. By Lemma \ref{Chat}, $\omega(\delta)>0$. There are two cases.
\begin{itemize}
\item If $\operatorname{dist}(\eta,\{-\eta_0,\eta_1\})<\delta$, then the first term is nonnegative, and $g(\eta,s,k,\mathfrak x)\geq M(\mathfrak x)-\ell(\delta+Q/(n^2h^2))$.
\item Otherwise the bracket is at least $\frac n2\mathfrak x_Jh^2\omega(\delta)-\sup|g|-O(h^{-2})$, which exceeds $M(\mathfrak x)$ for $h$ large.
\end{itemize}
Hence $\liminf_h\inf_J(\mathcal G_h-M)\geq-\ell\delta$ for every $\delta>0$.

\emph{Step 2: the value $M(\lambda)$.} The function $g(\eta,s,F_0,\lambda)=cn+s\,\kappa(\eta)$ is affine in $s$, with
\[
\kappa(\eta)=-c-\frac{F_0}2+n\Big(\frac{F_0}2-c\Big)\eta^2-\frac{F_0\lambda\eta}2 .
\]
Using $\lambda\eta_0=2$, $\lambda\eta_1=2/(n-1)$ and $F_0=2nc/(n+2)$,
\[
\kappa(-\eta_0)=-\frac{2nc}{n+2},\qquad \kappa(\eta_1)=-\frac{2n(n+1)c}{(n-1)(n+2)} .
\]
Therefore
\[
M(\lambda)=cn+\min\{0,\kappa(-\eta_0),\kappa(\eta_1)\}=
\begin{cases}\dfrac{n(n^2-n-4)}{(n-1)(n+2)},&c=1,\\[6pt] -n,&c=-1,\end{cases}
\]
and $M(\lambda)/\varpi=\mu_c$.

\emph{Step 3: the first zero.} Since $L$ is quadratic, $L(\lambda+t)=-dt-t^2/4$. The constants are chosen in the following order.
\begin{enumerate}
\item Fix $\iota\in(0,1)$ and $\delta\in(0,\min\{\lambda/2,\varpi\})$.
\item By Step 1 with $J=[\lambda/2,\lambda+\delta]$, there are $h_1$ and $C_1$ with $|\mathcal G_h|\leq C_1$ on $J$ for $h\geq h_1$.
\item Choose $\Upsilon>\max\{C_1/(\varpi-\delta/4),\ |\mu_c|+1\}$.
\item Choose $h_2\geq h_1$ such that, for $h\geq h_2$, all of the following hold: $\Upsilon/h^2<\delta$; $\sup_J|\mathcal G_h-M|<\varpi\iota/3$; $|M(\mathfrak x)-M(\lambda)|<\varpi\iota/3$ for $|\mathfrak x-\lambda|\leq \Upsilon/h^2$; $\Upsilon^2/(4h^2)<\varpi\iota/3$; and $L(\lambda-\delta)>C_2/h^2$, with $C_2$ as in the next paragraph.
\end{enumerate}
The last requirement uses the following bound. For $\mathfrak x\in(0,\lambda-\delta]$, the term $\frac n2\mathfrak x h^2(\hat C+o_n)$ is nonnegative, $g$ is bounded on $[-\eta_0,\eta_0]\times[0,1]\times I_{n,c}(h)\times(0,\lambda]$, and $\varepsilon$ is bounded. So $\mathcal G_h\geq-C_2$ there, and $\mathcal F_h\geq L(\lambda-\delta)-C_2/h^2>0$ on $(0,\lambda-\delta]$.

Let $h\geq h_2$ and write $\mathfrak x=\lambda+\tau/h^2$. Then $h^2\mathcal F_h(\mathfrak x)=-\varpi\tau-\tau^2/(4h^2)+\mathcal G_h(\mathfrak x)$.
\begin{itemize}
\item For $-\delta h^2\leq\tau\leq-\Upsilon$ we have $\tau^2/(4h^2)\leq\delta|\tau|/4$. Hence $h^2\mathcal F_h\geq(\varpi-\delta/4)|\tau|-C_1>0$.
\item For $|\tau|\leq \Upsilon$ we have $h^2\mathcal F_h(\mathfrak x)=-\varpi\tau+M(\lambda)+\vartheta$ with $|\vartheta|<\varpi\iota$. Since $M(\lambda)=\varpi\mu_c$, this is positive for $\tau\leq\mu_c-\iota$ and negative for $\tau=\mu_c+\iota$.
\end{itemize}
Thus $\mathcal F_h>0$ on $(0,\lambda+(\mu_c-\iota)/h^2]$ and $\mathcal F_h(\lambda+(\mu_c+\iota)/h^2)<0$. Since $\Pi(\mathfrak x h^2)=\mathfrak x^2h^6\mathcal F_h(\mathfrak x)$ and $x_{\rm opt}$ is the first positive zero of $\Pi$, we get $|x_{\rm opt}-\lambda h^2-\mu_c|\leq\iota$ for all $h\geq h_2$. Since $\iota$ is arbitrary, the theorem follows.
\end{proof}

\begin{proposition}\label{nonconvex}
Let $c=1$ and $n\geq3$. For all sufficiently large $h$ there are a trace-free symmetric $B$ with $|B|<x_{\rm opt}$ and a unit vector $e$ with the following property. For every $T\perp e$ with $|T|^2=s\in[0,1]$ and every $k\geq0$, the operator $A=h\Id+\frac{1}{2h}B+\frac{k}{2h}(T\otimes T-\frac sn\Id)$ satisfies $\ip{Ae}{e}<0$.
\end{proposition}
\begin{proof}
Let $x=\lambda h^2+\mu_1/2$. By Theorem \ref{main-asymp}(a) and $\mu_1>0$, $x<x_{\rm opt}$ for large $h$. Take $B$ with eigenvalue $-\eta_0x$ on $e$ and $\eta_1x$ on $e^\perp$. Since $\eta_0\lambda=2$,
\[
\ip{Ae}{e}=h-\frac{\eta_0x}{2h}-\frac{ks}{2nh}=-\frac{\eta_0\mu_1}{4h}-\frac{ks}{2nh}<0 .
\]
\end{proof}

\begin{remark}\label{opt-remark}
Proposition \ref{nonconvex} is a statement about algebraic data at a point; see Remark \ref{algebraic-remark}. For $c=-1$, $\mu_{-1}<0$; this reflects the ambient term $c(n-s)|B|^2$, which is unfavorable in that case.
\end{remark}

\subsection*{Asymptotics of \texorpdfstring{$R_*$}{R*} and comparison with convexity}
We keep the notation $\eta_0$, $o_n$, $\lambda$ and $\varpi$ introduced above, and put $p_0=\frac n2o_n=\frac{n(n-2)}{2\sqrt{n(n-1)}}$. Then
\begin{equation}\label{dp0}
\varpi^2=p_0^2+n,\qquad \varpi-p_0=\frac1{\eta_0}=\frac\lambda2 .
\end{equation}
Indeed, $\varpi^2-p_0^2=\frac{n^3-n(n-2)^2}{4(n-1)}=n$ and $\varpi-p_0=\frac{\sqrt n\,(n-(n-2))}{2\sqrt{n-1}}=\sqrt{\frac n{n-1}}$.

\begin{proof}[Proof of Theorem \ref{main-asymp}{(b)--(c)}]
Let $h\to\infty$; then $h$ lies in \eqref{gap-range}. In this proof $C$ denotes constants depending only on $n$, and every $O(\cdot)$ is uniform in $s\in[0,1]$.

\emph{Step 1: the coefficient.} By Lemma \ref{bounds}(d), $F_0-Qw^2\leq F_c<F_0$ for $r>0$. Since $w\leq1/(nh)$ on $(0,r_0]$,
\begin{equation}\label{k-uniform}
0\leq F_0-k(s)\leq\frac{Q}{n^2h^2}\qquad(0\leq s\leq1).
\end{equation}
By Lemma \ref{bounds}(e), $|sk'|\leq\sigma_{n,c}(h)=O(h^{-2})$. By \eqref{alphaeta}--\eqref{Lambda} and $|k|\leq \mathcal K$, $0\leq\Lambda\leq C$.

\emph{Step 2: expansion of $q_{n,c,h}$ and $\mathcal P_{n,c,h}$.} By \eqref{ab},
\[
ab=h^2+\frac{(n-2)ks}{2n}-\frac{(n-1)k^2s^2}{4n^2h^2},\qquad
b^2=h^2-\frac{ks}{n}+\frac{k^2s^2}{4n^2h^2}.
\]
Hence \eqref{qPdef}, Step 1 and \eqref{k-uniform} give
\[
q_{n,c,h}(s)=nh^2+\mathfrak q_c(s)+O(h^{-2}),
\]
where
\[
\mathfrak q_1(s)=n(1-s)+\frac{(n-2)F_0}{2}s,\qquad
\mathfrak q_{-1}(s)=-n+s-\frac{F_0}{2}s .
\]
For $c=-1$ we used $ab+(n-1)b^2=nh^2-\frac{ks}2+O(h^{-2})$. Similarly
\[
\mathcal P_{n,c,h}(s)=p_0+\frac{\gamma(s)}{h^2}+O(h^{-4}),\qquad \gamma(s)=\frac{|F_0|\,\eta_0\,s}{2}.
\]

\emph{Step 3: the threshold $R_*$.} Let $R_h(s)$ be the expression inside the minimum in \eqref{radius}. Write $\mathcal P_{n,c,h}=p_0+\epsilon_1$ and $q_{n,c,h}/h^2=n+\epsilon_2$, where $\epsilon_1=\gamma/h^2+O(h^{-4})$ and $\epsilon_2=\mathfrak q_c/h^2+O(h^{-4})$. Since $p_0^2+n=\varpi^2>0$ by \eqref{dp0}, Taylor's formula for the square root gives
\[
\sqrt{(p_0+\epsilon_1)^2+n+\epsilon_2}=\varpi+\frac{2p_0\epsilon_1+\epsilon_2}{2\varpi}+O(h^{-4}).
\]
Therefore, using \eqref{dp0},
\begin{align*}
R_h(s)&=2h^2\Big(\varpi-p_0+\frac{2p_0\epsilon_1+\epsilon_2}{2\varpi}-\epsilon_1\Big)+O(h^{-2})\\
&=\lambda h^2+\frac{\mathfrak q_c(s)-2\gamma(s)(\varpi-p_0)}{\varpi}+O(h^{-2}),
\end{align*}
that is,
\[
R_h(s)=\lambda h^2+\frac{\psi_c(s)}{\varpi}+O(h^{-2}),\qquad \psi_c(s)=\mathfrak q_c(s)-|F_0|\,s .
\]
Since the error is uniform in $s$, $R_*=\min_sR_h(s)=\lambda h^2+\min_{[0,1]}\psi_c/\varpi+O(h^{-2})$.

For $c=1$, $F_0=2n/(n+2)$ and $\psi_1(s)=n-\frac{6n}{n+2}s$, which is decreasing. Its minimum is $\psi_1(1)=\frac{n(n-4)}{n+2}$, and
\[
\frac{\psi_1(1)}\varpi=\frac{n(n-4)}{n+2}\cdot\frac{2\sqrt{n-1}}{n\sqrt n}=\frac{2(n-4)(n-1)}{(n+2)\sqrt{n(n-1)}}=\rho_1 .
\]
For $c=-1$, $F_0=-2n/(n+2)$ and $\psi_{-1}(s)=-n+\frac{2}{n+2}s$, which is increasing. Its minimum is $\psi_{-1}(0)=-n$, and $-n/\varpi=-2\sqrt{(n-1)/n}=\mu_{-1}$.

\emph{Step 4: the convexity level.} For $c=1$, $k\geq0$ by Lemma \ref{bounds}(d), so $a-b=\frac{ks}{2h}\geq0$ and $\mathfrak m=b=h-\frac{ks}{2hn}$. By \eqref{k-uniform}, $0\leq ks\leq F_0$ and $k(1)\geq F_0-Q/(n^2h^2)$. Hence $\min_{[0,1]}\mathfrak m=h-\frac{F_0}{2hn}+O(h^{-3})$ and
\begin{align*}
X_{\rm conv}&=\frac{2h}{\eta_0}\min_{[0,1]}\mathfrak m=\lambda h^2-\frac{F_0}{n\eta_0}+O(h^{-2})\\
&=\lambda h^2-\frac{2\sqrt n}{(n+2)\sqrt{n-1}}+O(h^{-2}).
\end{align*}
For $c=-1$, $k<0$, so $a\leq b$ and $\mathfrak m=a=h-\frac{(n-1)|k|s}{2hn}$. By \eqref{k-uniform}, $|F_0|\leq|k|\leq|F_0|+Q/(n^2h^2)$, so $\max_s|k|s=|F_0|+O(h^{-2})$. Hence
\[
X_{\rm conv}=\lambda h^2-\frac{(n-1)|F_0|}{n\eta_0}+O(h^{-2})=\lambda h^2-\frac{2\sqrt{n(n-1)}}{n+2}+O(h^{-2}).
\]

\emph{Step 5: consequences.} For $c=1$,
\[
\rho_1+\frac{2\sqrt n}{(n+2)\sqrt{n-1}}=\frac{2[(n-4)(n-1)+n]}{(n+2)\sqrt{n(n-1)}}=\frac{2(n-2)^2}{(n+2)\sqrt{n(n-1)}},
\]
and, by part (a),
\[
\mu_1-\rho_1=\frac{2[(n^2-n-4)-(n-4)(n-1)]}{(n+2)\sqrt{n(n-1)}}=\frac{8(n-2)}{(n+2)\sqrt{n(n-1)}} .
\]
Thus $R_*>X_{\rm conv}$ for all large $h$, consistently with Theorem \ref{main-compare}(ii). For $c=-1$,
\[
-\frac{2\sqrt{n(n-1)}}{n+2}-\mu_{-1}=2\sqrt{n-1}\Big(\frac1{\sqrt n}-\frac{\sqrt n}{n+2}\Big)=\frac{4\sqrt{n-1}}{(n+2)\sqrt n}>0,
\]
and $x_{\rm opt}-R_*=\mu_{-1}-\mu_{-1}+o(1)=o(1)$ by part (a). The sign of $X_{\rm conv}-R_*$ for $c=-1$ agrees with Theorem \ref{main-compare}(i), which holds for every $h>1$.
\end{proof}

\section*{Funding}
H. Alencar and M. Batista were partially supported by the Brazilian National Council for Scientific and Technological Development -- CNPq, Brazil (grants 303118/2022-9 and 304381/2026-8, respectively), and by the Coordination for the Improvement of Higher Education Personnel -- CAPES, Brazil (Finance Code 001).

\section*{Data availability statement}
This manuscript has no associated data.
\section*{Conflict of interest statement}
The authors declare that they have no conflict of interest.

\section*{AI disclosure}
A generative AI tool was used to assist with language editing, organization of the exposition, and discussion and checking of selected algebraic calculations. The authors take full responsibility for the mathematical statements, proofs, and references.

\end{document}